\documentclass[smallextended]{svjour3}

\smartqed
\usepackage{graphicx}
\usepackage{amssymb,amsmath}
\usepackage{amsfonts}
\usepackage[colorlinks=true,citecolor=blue]{hyperref}
\usepackage{pdfsync}
\usepackage{enumitem}
\usepackage[TABBOTCAP,tight]{subfigure}
\newcounter{nameOfYourChoice}
\newcommand{\Real}{\mathbb{R}}                              
\newcommand{\set}[1]{\left\{#1\right\}}                     
\newcommand{\abs}[1]{\left|#1\right|}                       
\newcommand{\bra}[1]{\left(#1\right)}
\newcommand{\brac}[1]{\left[#1\right]}                    
\newcommand{\blue}[1]{#1}                    
\newcommand{\bblue}[1]{#1}

\newcommand {\beq}{\begin{equation}}
\newcommand {\eeq}{\end{equation}}
\newcommand {\beqn}{\begin{equation*}}
\newcommand {\eeqn}{\end{equation*}}
\newcommand {\bear}{\begin{eqnarray}}
\newcommand {\eear}{\end{eqnarray}}
\newcommand {\bearn}{\begin{eqnarray*}}
\newcommand {\eearn}{\end{eqnarray*}}
\usepackage{relsize}

\newcommand{\norm}[1]{\left\lVert#1\right\rVert}

\usepackage{scalerel}
\DeclareMathOperator*{\cart}{\scalerel*{\times}{\sum}}

\DeclareMathOperator{\conv}{conv}

\DeclareMathOperator{\cone}{cone}
\DeclareMathOperator{\spann}{span}
\DeclareMathOperator{\cl}{cl}
\DeclareMathOperator{\ext}{ext}
\DeclareMathOperator{\bd}{bd}

\DeclareMathOperator{\Int}{int}
\DeclareMathOperator{\epi}{epi}
\DeclareMathOperator{\rank}{rank}

\DeclareMathOperator{\aff}{aff}
\DeclareMathOperator{\relint}{ri}
\DeclareMathOperator{\relbd}{rbd}

\DeclareMathOperator{\dom}{dom}

\newcommand {\e}{{\bf e}}

\newcommand {\bcom}{}
\usepackage{stmaryrd}
\newcommand{\sidx}[1]{\left\llbracket     #1 \right\rrbracket}

\usepackage{thmtools,thm-restate}
\spnewtheorem{assumption}{Assumption}{\bf}{\it}

\begin{document}

\title{Small and Strong Formulations for Unions of Convex Sets from the Cayley Embedding}
\titlerunning{Small and Strong Formulations for Unions of Convex Sets}
\author{Juan Pablo Vielma}
\authorrunning{Juan Pablo Vielma} 
\institute{Sloan School of Management, Massachusetts Institute of Technology, Cambridge MA 02139, USA,\\
\email{jvielma@mit.edu}}
\date{Received: date / Accepted: date}

\maketitle              

\begin{abstract}
There is often a significant trade-off between formulation strength and size in mixed integer programming (MIP).
When modeling convex disjunctive constraints (e.g. unions of convex sets), adding auxiliary continuous variables can sometimes help resolve this trade-off. However, \bblue{standard formulations that use such auxiliary continuous variables can have a worse-than-expected computational effectiveness, which is often attributed precisely to these auxiliary continuous variables}. For this reason, there has been considerable interest in constructing strong formulations that do not use continuous auxiliary variables. We introduce a technique to construct formulations without these detrimental continuous auxiliary variables. To develop this technique we introduce a natural  non-polyhedral generalization of the Cayley embedding of a family of polytopes and show it inherits many geometric properties of the original embedding.  We then show how the associated formulation technique can be used to construct small and strong formulation for a wide range of disjunctive constraints. In particular, we show it can recover and generalize all known strong formulations without continuous auxiliary variables.

\end{abstract}
\pagebreak
\section{Introduction}
Mixed integer programming (MIP) adds integrality requirements to \blue{a continuous optimization problem, which is often referred to as the continuous relaxation of the MIP problem}. \blue{MIP problems with a convex continuous relaxation}  often arise from the need to model disjunctive constraint of the form $x\in \bigcup\nolimits_{i=1}^k C^i$ where $\{C^i\}_{i=1}^k\subseteq \mathbb{R}^n$ is a family of closed convex sets. The two main classes of formulations for these constraints are the so-called \emph{Big-M} and \emph{convex hull} formulations. Big-M formulations are simple and small, but their continuous relaxations usually yield weak bounds, which can  hinder the performance of branch-and-bound based algorithms. In contrast, the convex hull formulation yields the best possible relaxation bounds for a single disjunctive constraint and normally yields strong bounds for problems with multiple constraints. Unfortunately, while convex hull formulations are only moderately larger than Big-M formulations, their computational performance is usually much worse. The folklore attributes this poor performance to certain continuous auxiliary variables used by the convex hull formulation. This has prompted significant interest on techniques to project out such variables (i.e. eliminate them without decreasing the formulation's strength). The resulting formulations can provide a significant computational advantage, but existing techniques are limited to very specific structures (e.g. see \cite{balas88,blair90,jeroslow88,Embedding,DBLP:journals/mp/VielmaN11} for polyhedra and \cite{lodi15,gunluk2010perspective,Hijazi,mohit1,mohit2} for  non-polyhedral sets). In this paper we introduce a technique  to project out the detrimental auxiliary variables  for a wide range of disjunctive constraints. In particular, the technique can be used to recover, and generalize all known results that use binary variables that add  to one (e.g. excluding the logarithmic formulation from \cite{Embedding,DBLP:journals/mp/VielmaN11}).

Our technique is based on a geometric characterization of the projection of the convex hull formulation that connects it to a natural non-polyhedral generalization of an object known as the Cayley embedding of a family of polytopes. To obtain this characterization we generalize to the non-polyhedral setting some known properties of the Cayley embedding and use it to obtain  a valid formulation of the disjunctive constraint. We then give simple sufficient conditions for this formulation to be equal to the projection of the convex hull formulation. Using these conditions we then recover and generalize all known techniques to project  the convex hull formulation. We also provide precise necessary and sufficient conditions to obtain the projection and comment on the practical implementation of the formulations. In particular, we evaluate the representation of the projection from an algebraic geometry perspective.

The paper is organized as follows. In Section~\ref{extended} we introduce a geometric abstraction that unifies all known formulations in a common framework. In Section~\ref{cayley} we introduce the geometric characterization and describe the projected convex hull formulation for two simple cases. In Section~\ref{boundary} we present the generalized properties of the Cayley embedding and the simple sufficient conditions. We then use these  conditions to recover and generalize all existing formulations in Section~\ref{applicasec}. In particular, we give  guidance on how to apply the technique, comment on its practical implementation and  present the algebraic geometry result. Finally, in Section~\ref{necandsufsec} we give detailed necessary and sufficient conditions for the technique. Omitted proofs are included in Section~\ref{proofsec}.

We use the following notation. For a function $f:\Real^n \to \Real\cup\set{\infty}$ we let its epigraph be $\epi\bra{f}:=\set{\bra{x,z}\in \Real^{n+1}\,:\, f(x)\leq z}$.  For a set $S\subseteq \Real^n$ we denote its topological closure, its convex hull, its conic hull and its affine hull by $\cl\bra{S}$, $\conv\bra{S}$, $\cone\bra{S}$ and $\aff\bra{S}$. For a closed convex  set \blue{$C\subseteq \Real^n$} we \blue{denote} its recession  cone \blue{by} $C_\infty$ \blue{ and the set containing all its extreme points by $\ext\bra{C}$}. We  let $\sidx{k}:=\set{1,\ldots,k}$,  $\e^i\in \Real^n$ be the $i$-th canonical vector,  ${\mathbf 1}\in\Real^n$ be the all ones vector and ${\mathbf 0}\in\Real^n$ be the all zeros vector (the specific dimension will be apparent from the context). For a closed convex cone $K$ we let $K^*$ be its polar cone. Finally, we let $\mathbb{Z}$ be the set of integers.
\section{MIP formulations for unions of convex sets}\label{extended}

\begin{definition}Let $\{C^i\}_{i=1}^k\subseteq \mathbb{R}^n$ be a finite family of closed convex sets and   $Q\subseteq \Real^{n+p+k}$ be a  closed convex set. We say $\bra{x,z,y}\in Q,\quad y\in \mathbb{Z}^k$ is a MIP formulation  of $x\in \bigcup\nolimits_{i=1}^k C^i$ if and only if
\beq\label{formcondition}x\in \bigcup\nolimits_{i=1}^k C^i \quad \Leftrightarrow\quad \exists\bra{z,y}\in \Real^{p}\times\mathbb{Z}^k \text{ s.t. }\bra{x,z,y}\in Q.\eeq
We refer to  $Q$ as the \emph{continuous relaxation} of the MIP formulation and say the formulation is \emph{ideal} if and only if for any minimal face $F$ of $Q$ and  $\bra{x,z,y}\in F$ we have $y\in \mathbb{Z}^k$.
\end{definition}

Existing formulations depend on specific set-representations. For instance, Balas, Jeroslow and Lowe give linear MIP formulations for polyhedra (e.g.  \cite[Section 5]{Mixed-Integer-Linear-Programming-Formulation-Techniques}),  Ben-Tal, Helton,  Nemirovski and Nie give conic MIP formulations for conic representable sets \cite{ben2001lectures,Helton} and  Ceria, Merhotra, Soares and Stubs give perspective function formulations for function level sets \cite{springerlink:10.1007/s101070050106,springerlink:10.1007/s101070050103}. To abstract the representation we use the \blue{following properties of the} gauge of a convex set (e.g. \cite{hiriart-lemarechal-2001}).
\blue{
\begin{lemma}\label{gaugepropertiesprop}Let $C\subseteq \Real^n$ be a closed convex set with ${\bf 0}\in C$ and $\gamma_C:\Real^n \to \Real \cup\set{\infty}$ be the gauge function of $C$ given by $\gamma_C\bra{x}:=\inf\set{\lambda>0\,:\, x\in \lambda C}$. Then the following properties hold.
  \begin{itemize}
    \item If $x\notin \lambda C$ for all $\lambda>0$, then  $\gamma_C\bra{x}=\infty$.
    \item The gauge function $\gamma_C$ is convex and positively homogeneous.
    \item We have that $\set{x\in \Real^n\,:\, \gamma_C\bra{x}\leq r}=rC$ and $\set{x\in \Real^n\,:\, \gamma_C\bra{x}\leq 0}=C_\infty$.
  \end{itemize}
Furthermore, if  $\tilde{C}\subseteq \Real^n$ is a closed convex set and $b\in \tilde{C}$, then \[\tilde{C}=\set{x\in \Real^n\,:\, \gamma_{\tilde{C}-b}\bra{x-b}\leq 1}.\]
\end{lemma}}

 Using gauge functions we can construct a generic versions of standard  formulations for  convex sets that satisfy the following assumption.

\begin{definition}\label{pointedrecessionassumption} We say $\mathcal{C}:=\set{C^i}_{i=1}^k\in \mathbb{C}_n$ if and only if $C^i\subseteq \Real^n$ is a non-empty closed convex set for all $i\in \sidx{k}$ and $C^i_\infty=C^j_\infty$ for all $i,j\in \sidx{k}$.
\end{definition}

\begin{restatable}{theorem}{extendedTheoremthm}\label{extendedTheorem} Let $\set{b^i}_{i=1}^k\subseteq \Real^n$ and $\mathcal{C}:=\set{C^i}_{i=1}^k\in \mathbb{C}_n$ be such that $b^i\in C^i$ for all $i\in \sidx{k}$, then an ideal formulation for $\bigcup_{i=1}^k C^i $ is given by
\begin{subequations}\label{extendedformulation}
\begin{alignat}{3}
\gamma_{C^i-b^i}\bra{x^i-b^i y_i}\leq y_i, \quad\forall  i\in \sidx{k}, \quad
  \sum\nolimits_{i=1}^k y_i&=1, \quad y\in \set{0,1}^k.\\
  \sum\nolimits_{i=1}^k x^i &=x.
\end{alignat}
\end{subequations}
In particular, if $\ext\bra{C^i}\neq \emptyset $ for all $i\in \sidx{k}$, then the continuous relaxation of \eqref{extendedformulation} is line-free and all of its extreme points have integral $y$ components.
\end{restatable}
The proof of Theorem~\ref{extendedTheorem} is analogous to existing formulations, but for completeness we include a proof in Section~\ref{Moreongaugesec}.

A key to obtain the relatively simple and small ideal formulation \eqref{extendedformulation} is the use of the $k$ copies $x^i$ of the original variables $x$. \bblue{Unfortunately, these variable copies induce a block structure that the MIP folklore identifies as a source of the worse-than-expected computational performance of formulation \eqref{extendedformulation}}. For this reason, simpler Big-M formulations are often preferred in practice, even though they usually fail to be ideal. We can abstract the specific structure of such Big-M formulations using gauge functions as follows.
\begin{restatable}{theorem}{BigMTheoremthm}\label{BigMextendedTheorem} Let $\set{b^i}_{i=1}^k\subseteq \Real^n$ and $\mathcal{C}:=\set{C^i}_{i=1}^k\in \mathbb{C}_n$ be such that $b^i\in C^i$ for all $i\in \sidx{k}$, and $M\in \Real^{k\times k}$ be such that $M_{i,i}=1$ for all $i\in \sidx{k}$ and
$C^j\subseteq \set{x\in \Real^n\,:\, \gamma_{C^i-b^i}\bra{x-b^i } \leq M_{i,j}}$
for all $i,j\in \sidx{k}$. Then a formulation for $\bigcup_{i=1}^k C^i $ is given by
\begin{equation}\label{bigMformulation}
\gamma_{C^i-b^i}\bra{x-b^i }\leq \sum\nolimits_{j=1}^k M_{i,j} y_j, \;\forall i\in \sidx{k},\quad
  \sum\nolimits_{i=1}^k y_i=1,\; y\in \set{0,1}^k.
\end{equation}
The strength of this formulation depended on $M$ with the strongest formulation being obtained for the smallest valid coefficients.
\end{restatable}

The abstraction provided by the use of gauge functions allow us to focus on the geometric structure of the formulations. However, it does not provide an explicit representation of the formulations that can be easily fed to a MIP solver. Fortunately, we can use known properties of gauge functions to obtain practical representation of various classes and recover existing formulations.

\begin{restatable}{lemma}{gaugelemmalem}\label{gaugelemma} Let $C\subseteq \Real^n$ be closed and convex with ${\bf 0}\in C$ and $E\subseteq \Real^{n}\times \Real_+$ be a closed convex cone. Then $E=\epi\bra{\gamma_C}$ if and only if $C=\set{x\in \Real^n\,:\, \bra{x,1}\in E}$ and  $C_\infty=\set{x\in \Real^n\,:\, \bra{x,0}\in E} $. In particular,
\begin{enumerate}
\item\label{gaugelemmaone} If $C:=\set{x\in \Real^n\,:\, \exists z\in \Real^p\text{ s.t. } A x +B z +c \in K}$ for a closed convex cone $K\subseteq \Real^m$, matrices $A\in \mathbb{R}^{m\times n}$ and $B\in \mathbb{R}^{m\times p}$,   and vector $c\in \Real^m$, then $\epi\bra{\gamma_C}=\set{\bra{x,y}\in \Real^{n+1}\,:\, \exists z\in \Real^p\text{ s.t. } A x +B z +c y \in K,\quad y\geq 0}$,
\item\label{gaugelemmathree} if $f:\Real^n \to \Real\cup \set{\infty}$ is a closed convex function, $(\cl \tilde{f})(x,y)$ is the closure of the perspective function of $f$, and $C:=\set{x\in \Real^d\,:\, f(x)\leq 0}$, then  $\epi\bra{\gamma_C}=\{(x,y)\in \Real^d\times \Real_+\,:\, (\cl \tilde{f})(x,y)\leq 0\}$,
\item if $b\in C$, then $\epi\bra{\gamma_{C-b}}=\set{\bra{x,y}\in \Real^{n+1}\,:\, \gamma_{C}\bra{x+b y}\leq y}$, and
\item \label{propintersect}If ${\bf 0}\in C^1\cap C^2$ then, $\epi\bra{\gamma_{C^1\cap C^2}}=\epi\bra{\gamma_{C^1}}\cap \epi\bra{\gamma_{C^2}}$.
\end{enumerate}
\end{restatable}

The following example illustrates Lemma~\ref{gaugelemma} and Theorems~\ref{extendedTheorem} and \ref{BigMextendedTheorem}.
\begin{example}\label{example1}
Let $C^1=\{x\in \Real^2\,:\, \bra{2- s_1  x_1 }\bra{2- s_2  x_2 }\geq 1 \quad \forall s\in \set{-1,1}^2\}$ and $C^2=[-5/4,5/4]^2$ \blue{be the sets depicted in Figure~\ref{ex1fig}}. Using standard conic representability results for the cone $\mathbf{L}^n:=\set{\bra{x,x_0}\in \mathbb{R}^{n+1}\,:\,\norm{x}_2\leq x_0}$ (e.g. \cite{ben2001lectures}) we have
\begin{alignat*}{3}
C^1
& =\set{x\in \Real^2\,:\, \norm{\bra{2, s_1 x_1-s_2 x_2}}_2\leq 4-s_1x_1-s_2x_2 \quad \forall s\in \set{-1,1}^2}.
\end{alignat*}
Using Lemma~\ref{gaugelemma}  we have $\gamma_{C^1}\bra{x}\leq y$ if and only if
\[\norm{\bra{2y , s_1 x_1-s_2 x_2}}_2\leq 4y-s_1 x_1-s_2 x_2  \quad \forall s\in \set{-1,1}^2\]
Similarly $\gamma_{C^2}\bra{x}\leq y$ if and only if
$-\bra{5/4}y \leq x_j\leq \bra{5/4}y \quad \forall j\in \sidx{2}$. Then Theorem~\ref{extendedTheorem} yields the ideal formulation for $\bigcup_{i=1}^2 C^i $ given by
\begin{subequations}\label{extendedformulationfirstex}
\begin{alignat}{5}
\norm{ \bra{2 y_1, s_1 x_1^1-s_2 x_2^1 }}_2&\leq  4y_1 - s_1x_1^1 - s_2x_2^1 &\quad& \forall s\in \set{-1,1}^2\\
-\bra{5/4}y_2 \leq x_j^2&\leq \bra{5/4}y_2  &\quad&  \forall j\in \sidx{2}\\
x^1+x^2 =x,\quad y_1+y_2=1,\quad y &\in \set{0,1}^2. &\quad&
\end{alignat}
\end{subequations}
Alternatively, we can use Theorem~\ref{BigMextendedTheorem} to obtain the formulation given by
\begin{subequations}\label{extendedformulationfirstexbigm}
\begin{alignat}{5}
\norm{ \bra{2 \bra{y_1+M_{1,2} y_2}, s_1 x_1-s_2 x_2 }}_2&\leq  4\bra{y_1+M_{1,2} y_2}\notag\\&\quad - s_1x_1 - s_2x_2 &\quad& \forall s\in \set{-1,1}^2\\
-M_{2,1}y_1-\bra{5/4}y_2 \leq x_j&\leq M_{2,1}y_1+\bra{5/4}y_2  &\quad&  \forall j\in \sidx{2}\\
 y_1+y_2=1,\quad y &\in \set{0,1}^2. &\quad&
\end{alignat}
\end{subequations}
The smallest Big-M values that make this formulation valid are $M_{1,2}=5/4$ and $M_{2,1}=3/2$. Unfortunately, we can check that for all $t\in (0,1)$ the point $\bra{\bar{x}(t),\bar{y}(t)}$ with $\bar{x}(t)=\bra{(5 + t)/4,(5/4)(5 - t) (t-1)/(3t-5)}$ and $\bar{y}(t)=\bra{t,1-t}$ is an extreme point of the continuous relaxation of \eqref{extendedformulationfirstexbigm} with  fractional $y$ components. Furthermore, $\bar{x}(t)\notin \conv\bra{C^1\cup C^2}$ for all $t\in (0,1)$.\qed

\end{example}
\begin{figure}[htpb]
\centering
\includegraphics[scale=.5]{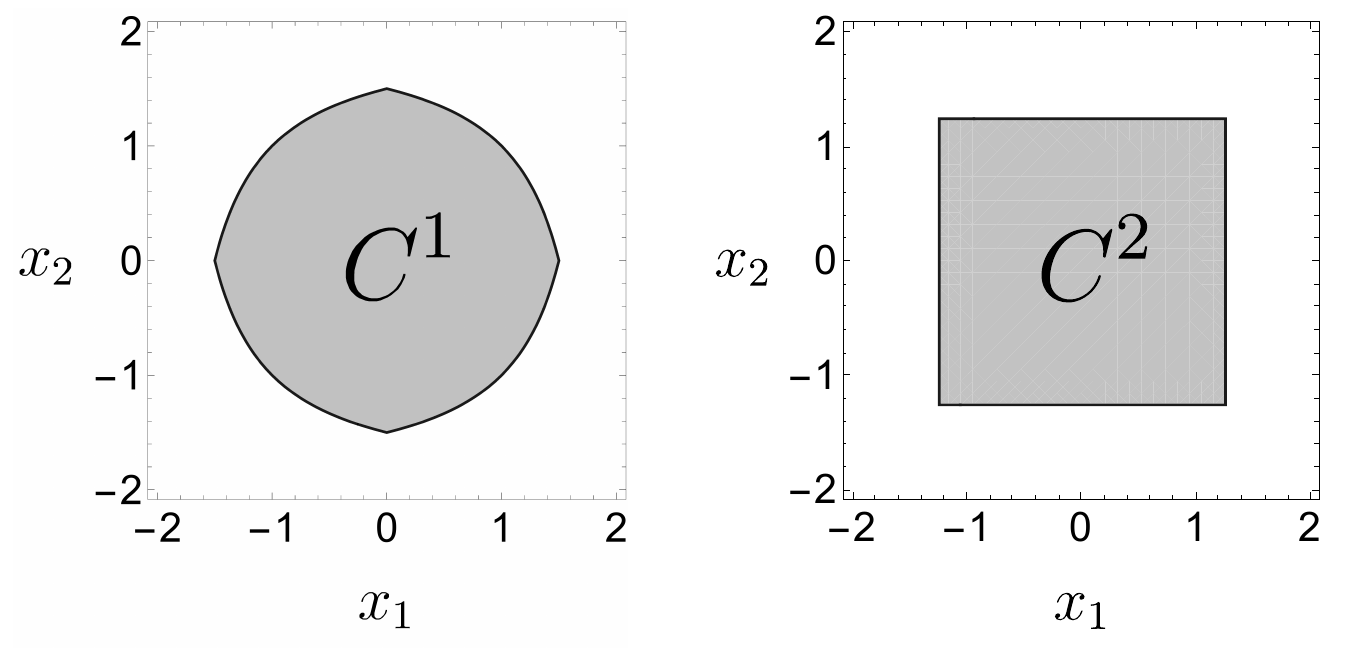}
\caption{Sets from Example~\ref{example1}.}\label{ex1fig}
\end{figure}
 An ideal formulation without the variable copies can be obtained by projecting \eqref{extendedformulation} onto the the $x$ and $y$ variables, but characterizing such projection can be challenging.  However, an effective characterization   can lead to significant computational improvements \cite{lodi15,gunluk2010perspective,Hijazi,Embedding,DBLP:journals/mp/VielmaN11}. Unfortunately, there are only few general 	techniques to obtain these characterizations.
One of the most general results by Balas, Blair and Jeroslow \cite{balas88,blair90,jeroslow88} considers unions of polyhedra with a common  geometric structure (See Proposition~\ref{blairprop} in Section~\ref{polysec}). In contrast, non-polyhedral results  require  more structure and fall into two classes. The first class considers convex sets contained in orthogonal spaces \cite{mohit2} and can be stated in our gauge notation as follows.
\begin{theorem}[\cite{mohit2}]\label{originalorthotheo}Let $\set{b^i}_{i=1}^k\subseteq \Real^n$, $\set{C^i}_{i=1}^k$ be a finite family of compact convex sets in $\Real^n$ and $\set{J_i}_{i=1}^k$ be disjoint sets such that $\bigcup_{i=1}^k J_i=\sidx{n}$ and for all $i\in \sidx{k}$ we have $b^i\in C^i$ and $C^i\subseteq \set{x\in \Real^n\,:\, x_j=0\quad \forall j\in  \sidx{n}\setminus J_i}$. Then an ideal formulation for $x\in \bigcup_{i=1}^k C^i$ is given by
\begin{equation}
\gamma_{C^i-b^i}\bra{ \brac{x- b^i y_i}_{J_i} }\leq y_i\quad\forall i\in \sidx{k},\quad
\sum\nolimits_{i=1}^k y_i=1,\quad y\in \set{0,1}^k.
 \end{equation}
  where for  $a \in \Real^n$ and $J\subseteq \sidx{n}$ we let $\brac{a}_J\in \Real^n$ be such that $\bra{\brac{a}_J}_j=a_j$ if $j\in J$ and $\bra{\brac{a}_J}_j=0$ otherwise.
\end{theorem}
The second class considers   sets with certain monotonicity properties and generalizes ``on/off'' constraints \cite{gunluk2010perspective,lodi15,Hijazi}.
\begin{theorem}[\cite{lodi15,Hijazi}]\label{isotone} Let $G^1,G^2\subseteq \Real^n$ be a closed convex sets such that $G^1_\infty= \Real^n_-$ and $G^2_\infty= \Real^n_+$.  Furthermore,  for each $i\in \sidx{2}$, let $l^i,u^i\in \Real^n$ and
$C^i:=\set{x\in G^i\,:\, l_j^i\leq x_j\leq u_j^i \quad \forall j\in \sidx{n}}$  be such that $l^i_j=\min\set{x_j\,:\, x\in C^i}$ and $u^i_j=\max\set{x_j\,:\, x\in C^i}$ for all  $j\in \sidx{n}$, $b^1=l^1$ and $b^2=u^2$.
Then an ideal formulation for $x\in C^1\cup C^2$ is given by
\begin{subequations}\label{isotoneform}
\begin{alignat}{3}
\gamma_{G^i-b^i}\hspace{-0.04in}\bra{\brac{x-l^1y_1- u^2 y_2}_{ J}}&\leq y_i,\; &\quad&\forall i\in \sidx{2},\;J\subseteq \sidx{n}\label{isotoneformexp}\\
y_1 l^1_j +y_2 l^2_j \leq x_j &\leq y_1 u^1_j +y_2 u^2_j,&\quad& \forall j\in \sidx{n}\\
y_1+y_2=1,\quad y&\in \set{0,1}^2.
 \end{alignat}
 \end{subequations}
\end{theorem}
The most general known version of this result (e.g. Theorem 4 in \cite{lodi15}) is obtained by combining Theorem~\ref{isotone} with Lemma~\ref{gaugelemma} and noting that the result is still valid if we flip or mirror the axes of the $x$ variables. In fact,  Theorems~\ref{isotone} and \ref{originalorthotheo} can also be easily extended further by combining formulation \eqref{isotoneform} and any orthogonal transformation of the $x$ variables (i.e. axis flip plus rotation).

\section{Ideal Formulations without Variable Copies}\label{cayley}

To construct the projection of formulation \eqref{extendedformulation} onto the $x$ and $y$ variables we use a geometric characterization introduced in \cite{Embedding} for the polyhedral setting.
This characterization is based on the \emph{Cayley trick} or \emph{Cayley Embedding}, which is used to study Minkowski sums of polyhedra (e.g. \cite{caytrick,karavelas2013maximum,WeibelPhd}). The characterization in \cite{Embedding} uses a generalization of the Cayley Embedding to consider alternative uses of $0$-$1$ variables (beyond the $k$ variables $y_i$ that add to one used in \eqref{extendedformulation}). However, for simplicity we only generalize the standard version  to the non-polyhedral setting  through the following result we prove in Section~\ref{Moreongaugesec}.
\begin{restatable}{proposition}{embpartprop}\label{embpart}  Let $\mathcal{C}:=\set{C^i}_{i=1}^k\in \mathbb{C}_n$ and $Q\bra{\mathcal{C}}:=\conv\bra{\bigcup\nolimits_{i=1}^k C^i\times \set{\e^i} }$,
where $\e^i$ is the i-th $k$-dimensional unit vector. Then
\begin{enumerate}
\item\label{embpart1} $Q\bra{\mathcal{C}}$ is a closed convex set and $Q\bra{\mathcal{C}}_\infty=\set{\bra{x,y}\in \Real^{n+k}\,:\,x\in C^1_\infty,\; y=0}$,
\item\label{embpart3} $Q\bra{\mathcal{C}}$ is the projection  of the continuous relaxation of \eqref{extendedformulation}, and
\item\label{embpart44} $\bra{x,y}\in Q\bra{\mathcal{C}},\quad y\in \mathbb{Z}^k$ is an ideal formulation of $x\in \bigcup_{i=1}^k C^i$.
\end{enumerate}
\end{restatable}

Proposition~\ref{embpart} reduces the construction of an ideal formulation to that of the convex hull defining $Q\bra{\mathcal{C}}$, which can be as challenging as the projection of \eqref{extendedformulation}. Fortunately, as illustrated in the following propositions, it can sometimes be easily constructed for special structures. The first structure we consider is \emph{nearly-homothetic} sets that are almost translations and scalings of one another (we replace the scaling by $0$ with the common recession cone of the sets).

\begin{proposition}\label{homothetict}
Let $C\subseteq \Real^n$  be a closed convex set such that ${\bf 0}\in C$,  $\set{b^i}_{i=1}^k\subseteq \mathbb{R}^n$ and $r\in \Real^k_+\setminus \set{{\bf 0}}$. If $\mathcal{C}:=\set{C^i}_{i=1}^k$ is such that $C^i=r_i C+b^i + C_\infty$ for each  $i\in \sidx{k}$, then $\bra{x,y}\in Q\bra{\mathcal{C}}$ if and only if
\begin{equation}\label{projectedgauge}
 \hspace{-0.1in} \gamma_{C}\left(x- \sum\nolimits_{i=1}^k y_i b^i\right)\leq \sum\nolimits_{i=1}^k r_i y_i,\quad \sum\nolimits_{i=1}^k y_i=1,\quad y_i\geq 0 \quad \forall i\in \sidx{k}.
\end{equation}
\end{proposition}
\begin{proof}
Let $Q$ be the \blue{set of points that satisfy}  \eqref{projectedgauge}. Then $Q$ is convex and $ C^i\times \set{\e^i}\subseteq Q$ for all $i\in \sidx{k}$, so we have $  Q\bra{\mathcal{C}}\subseteq Q$.
Finally, if $\bra{x,y}\in Q$, then $x\in \bra{\sum_{i=1}^k y_i r_i } C +\sum_{i=1}^k y_i b^i\subseteq \bra{\sum_{i=1}^k y_i r_i } C +\sum_{i=1}^k y_i b^i+C_\infty=\sum_{i=1}^k y_i \bra{r_i C +b^i+C_\infty}$ \blue{and hence $\bra{x,y}\in Q\bra{\mathcal{C}}$.} \qed
\end{proof}

The second structure we consider is a technical generalization of  Theorem~\ref{originalorthotheo} that we later use to generalize Theorem~\ref{isotone}. This generalization relaxes the orthogonality requirement of Theorem~\ref{originalorthotheo} by allowing sets that are the Minkowski sum of the orthogonal sets and the non-negative orthant. This requires adding technical restriction \eqref{genisocond} on the sets, which generalizes the monotonicity condition of \blue{Theorem~\ref{isotone}}. We discuss this condition further in \blue{Section~\ref{isotonegensec}}. Finally, we explicitly consider the possible orthogonal transformation we have previously alluded to, and which we represent through an orthonormal basis. This last step allows for a more direct practical application of the result, but makes the proof more technical so we postpone it to Section~\ref{proofofgeneralorthpropsec}.
\begin{proposition}\label{generalorthprop} Let $\set{G^i}_{i=1}^k$ be closed convex sets in $\Real^n$ such that ${\bf 0} \in G^i$ for all $i\in \sidx{k}$,   $\set{v^j}_{j=1}^n\subseteq \Real^n$ be an orthonormal basis of $\Real^n$, and $\set{J_i}_{i=1}^k$ be disjoint sets such that $\sidx{n}=\bigcup_{i=1}^k J_i$, $\set{b^i}_{i=1}^k\subseteq \Real^n$, $t\in \set{-1,1}^n$ and $M=\cone\bra{\set{ t_j v^j}_{j=1}^n}$. Finally, let $\set{s^i}_{i=1}^k\subseteq \set{-1,0,1}^n$ be such that
  for each $i\in \sidx{k}$  we have  $s^i_j=0$ for all $j\notin J_i$,  $K^i=\cone\bra{\set{ s^i_j v^j}_{j=1}^n}$ and  $C^i:=b^i+G^i\cap K^i+M$. If for all $i\in \sidx{k}$ we have
\begin{equation}\label{genisocond}
 \bra{\bra{G^i\cap K^i}-K^i}\cap K^i=G^i\cap K^i
 \end{equation}
 and $G^i\cap K^i$ is compact, then $\bra{x,y}\in Q\bra{\mathcal{C}}$ if and only if
\begin{subequations}\label{orthogonalplusprojcone}
\begin{alignat}{3}
\gamma_{G^i}\bra{\sum\nolimits_{j\in J_i} u^{i,j} \bra{u^{i,j}\cdot x-\sum\nolimits_{l=1}^k \overline{b}^{i,l}_j y_l }^{\hspace{-0.05in}\text{\relsize{2}{$+$}}\hspace{0.05in}}}&\leq y_i&\quad&\forall i\in \sidx{k}\label{orthogonalplusprojcone1}\\
t_j   v^j\cdot x-\sum\nolimits_{l=1}^k \underline{b}^{l}_j y_l &\geq 0 &\quad& \forall  j\in  \sidx{n}\label{orthogonalplusprojcone2}\\
 \sum\nolimits_{i=1}^k y_i=1,\quad y_i&\geq 0 &\quad& \forall i\in \sidx{k},
\end{alignat}
 \end{subequations}
 where  $\bra{a}^+=\max\set{0,a}$ for any $a\in \Real$ and for all $i,l\in \sidx{k}$ and $j\in \sidx{n}$ we let $u^{i,j}= \bra{-s_j^i t_j}^+ s_j^i v^j$, $\underline{b}^{i}_j=\min\set{t_j   v^j\cdot x\,:\, x\in b^i+G^i\cap K^i}$ and  $\overline{b}^{i,l}_j=\max\set{s_j^i   v^j\cdot x\,:\, x\in b^l+G^l\cap K^l}$ if $i\neq l$ and  $\overline{b}^{i,i}_j=u^{i,j}\cdot b^i$.
\end{proposition}

\section{Boundary Structure of the Cayley Embedding}\label{boundary}
To characterize $Q\bra{\mathcal{C}}$ for more complicated unions we will use the special structure of its boundary. It is known that if all $C^i$ are polytopes, then every face of $Q\bra{\mathcal{C}}$ is of the form $\conv\bra{\bigcup_{i=1}^n F^i\times\set{\e^i}}$ where  the $F^i$ are faces of $C^i$ whose normals intersect \cite{caytrick,karavelas2013maximum,WeibelPhd}. We generalize this result beyond polyhedra using standard properties of the boundary of a closed convex set (e.g. \cite{hiriart-lemarechal-2001}).

\begin{definition} The support function of $S\subseteq \Real^n$ is the function $\sigma_S:\Real^n \to \Real \cup \set{\infty}$ defined by $\sigma_S\bra{d}:=\sup\set{d\cdot x\,:\, x\in S}$.
The domain of $\sigma_S$ is the set $\dom\bra{\sigma_S}:=\set{d\in \mathbb{R}^n\,:\, \sigma_S\bra{d}< \infty}$.

For a closed convex set $C\subseteq \Real^n$ we denote its boundary by $\bd\bra{C}=C\setminus \Int\bra{C}$, its relative boundary by $\relbd\bra{C}=C\setminus \relint \bra{C}$, its affine hull by $\aff\bra{C}$ and the linear subspace parallel to $\aff\bra{C}$ by $L\bra{C}$.

The face of $C$ exposed by $d\in \Real^n$ is $F_C\bra{d}:=\set{x\in C\,:\, d\cdot x=\sigma_C\bra{d}}$ and its normal cone at $x\in \bd\bra{C}$ is $N_C(x):=\set{d\in \Real^n\,:\, d\cdot(y-x)\leq 0\quad \forall y\in C}$. The tangent cone $T_C(x)$ to $C$ at $x\in \bd\bra{C}$ is the polar of   $N_C(x)$.

\end{definition}

\begin{restatable}{proposition}{boundaryprop}\label{bdprop}
Let $\mathcal{C}:=\set{C^i}_{i=1}^k\in \mathbb{C}_n$ , $\set{b^i}_{i=1}^k\subseteq \Real^n$ be such that $b^i\in C^i$ for all $i\in \sidx{k}$ and $A\in \Real^{r\times n}$ be such that $L\bra{\mathcal{C}}:=\sum_{i=1}^k L\bra{C^i}=\set{x\in \Real^n\,:\, Ax={\bf 0}}$. Then
\[\aff\bra{Q\bra{\mathcal{C}}}=\set{\bra{x,y}\in \Real^{n+k}\,:\, \sum\nolimits_{i=1}^k y_i=1,\quad Ax=\sum\nolimits_{i=1}^k A b^i y_i}.\]
In addition, let $\mathcal{U}\bra{\mathcal{C}}:=\set{u\in L\bra{\mathcal{C}}\setminus \set{\bf 0}\,:\, F_{C^i}(u)\neq \emptyset \quad \forall i\in \sidx{k}}$,
$N\bra{\mathcal{C}}:=\set{\bra{{x}^i}_{i=1}^k\in \cart\nolimits_{i=1}^k \bd\bra{C^i}  \,:\, L\bra{\mathcal{C}}\cap\bigcap\nolimits_{i=1}^k N_{C^i}\bra{{x}^i} \neq \set{\bf 0}}$,
and for each $\mathcal{X}:=\bra{{x}^i}_{i=1}^k\in N\bra{\mathcal{C}}$ let $Q\bra{\mathcal{X}}:=\conv\bra{\bigcup_{i=1}^n \set{{x}^i}\times \set{\e^i} }$. Then  $\relbd\bra{Q\bra{\mathcal{C}}}$ is equal to the union of
\beq\label{relbdunion1}
\bigcup\nolimits_{u\in \mathcal{U}\bra{\mathcal{C}}} \conv\bra{\bigcup\nolimits_{i=1}^k F_{C^i}(u)\times \set{\e^i}}=\bigcup\nolimits_{\substack{\mathcal{X}\in N\bra{\mathcal{C}}}} Q\bra{\mathcal{X}}
\eeq
and
\beq\label{relbdunion2}
\bigcup\nolimits_{i=1}^k \conv\bra{\bigcup\nolimits_{j\neq i} C^j \times \set{\e^j}}= \bigcup\nolimits_{i=1}^k \set{\bra{x,y}\in Q\bra{\mathcal{C}}\,:\, y_i=0}.
\eeq
\end{restatable}

We postpone the proof of Proposition~\ref{bdprop} to Section~\ref{proofofbdprop} and instead illustrate it in the following example.

\begin{example}\label{example2}Let
  $C^1=\set{x\in \Real^2\,:\, \bra{2- v_1 x_1 }\bra{2- v_2 x_2 }\geq 1 \quad \forall v\in \set{-1,1}^2}$ and $C^2=[-5/4,5/4]^2$ be the sets from Example~\ref{example1} depicted in Figures~  \ref{ex1fig} and \ref{normalfigb}. Let  $\hat{x}^2=(1,-1)\in \bd\bra{C^2}$ and $\hat{B}^1:=\set{x\in C^1\,:\, \bra{2-x_1}\bra{2+x_2}=1}\subseteq \bd\bra{C^1}$ be the boundary subsets highlighted in black in Figure~\ref{normalfigb} (the range of their normals are depicted by dashed arrows). Then $N_{C^1}\bra{{x}^1}\cap N_{C^2}\bra{\hat{x}^2}\neq \emptyset$ if and only if ${x}^1\in \hat{B}^1$ and hence by Proposition~\ref{bdprop} we have that $\hat{B}:=\bigcup_{{x}^1\in \hat{B}^1}\conv\bra{\bra{\set{{x}^1}\times \set{\e^1}}\cup \bra{\set{\hat{x}^2}\times \set{\e^2}}}\subseteq \relbd\bra{Q\bra{\mathcal{C}}}$. This is illustrated in Figure~\ref{normalfiga} were we use the fact that $y_1+y_2=1$ for all  $\bra{x,y}\in Q\bra{\mathcal{C}}$ to eliminate $y_2$ and depict $Q\bra{\mathcal{C}}$ three dimensions. In Figure~\ref{normalfiga} the representations (i.e. after eliminating $y_2$) of  $\hat{B}^1\times \set{\e^1}$ and $\set{\hat{x}^2}\times \set{\e^2}$ are highlighted in black and $\hat{B}$ corresponds to the meshed  surface. This surface is an example of a portion of the boundary of $Q\bra{\mathcal{C}}$ considered in \eqref{relbdunion1}. We obtain another  example of this portion if we let $\tilde{x}^1:=\bra{0,-3/2}\in\bd\bra{C^1}$ and $B^2:=\set{x\in C^2\,:\,x_2=-5/4}\subseteq \bd\bra{C^2}$ be the boundary subsets highlighted in white in Figure~\ref{normalfigb}, for which $N_{C^1}\bra{\tilde{x}^1}\cap N_{C^2}\bra{x^2}\neq \emptyset$ if and only if $x^2\in B^2$, and $\tilde{B}:=\conv\bra{\bra{\set{\tilde{x}^1}\times \set{\e^1}}\cup \bra{B^2\times \set{\e^2}}}\subseteq \relbd\bra{Q\bra{\mathcal{C}}}$.  An example of a portion considered in \eqref{relbdunion2} is simply $C^1 \times \set{\e^1}$ whose representation is depicted by the dotted surface in  Figure~\ref{normalfiga}. \qed
\end{example}
\begin{figure}[htpb]
\centering
\subfigure[Cayley embedding with variable $y_2$ projected out.]{\includegraphics[scale=0.55]{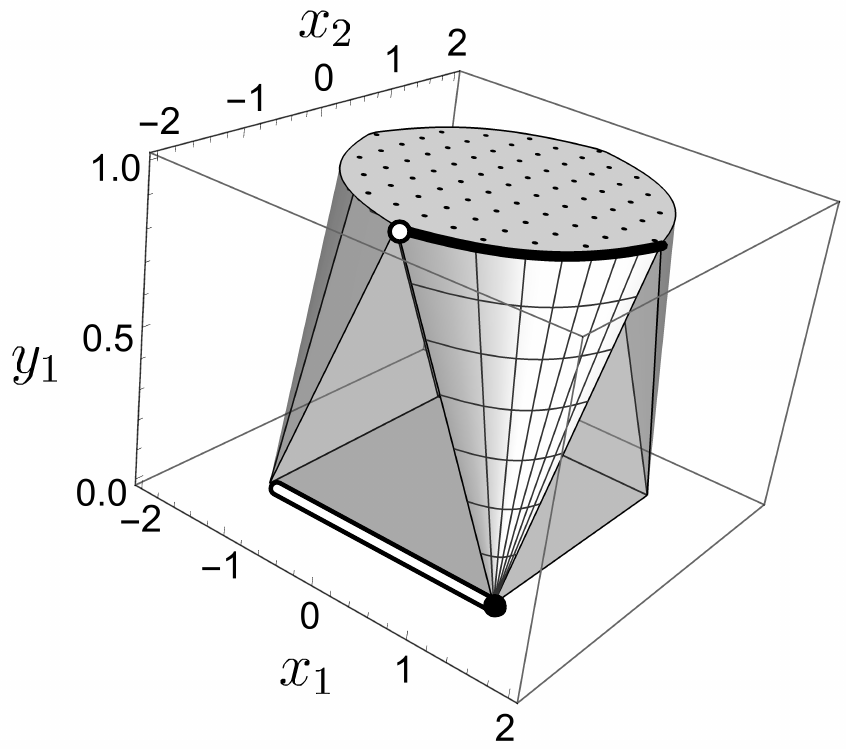}\label{normalfiga}}
\subfigure[Sets in dark gray. Boundary subsets from Example~\ref{example2} in black and white, with normals as dashed arrows. Supersets $C^{s,i}$ from Example~\ref{example2b} in light gray. One nonlinear inequality of $C^1$ as dotted curve.]{\includegraphics[scale=.35,trim = -4mm -10mm -4mm 0mm]{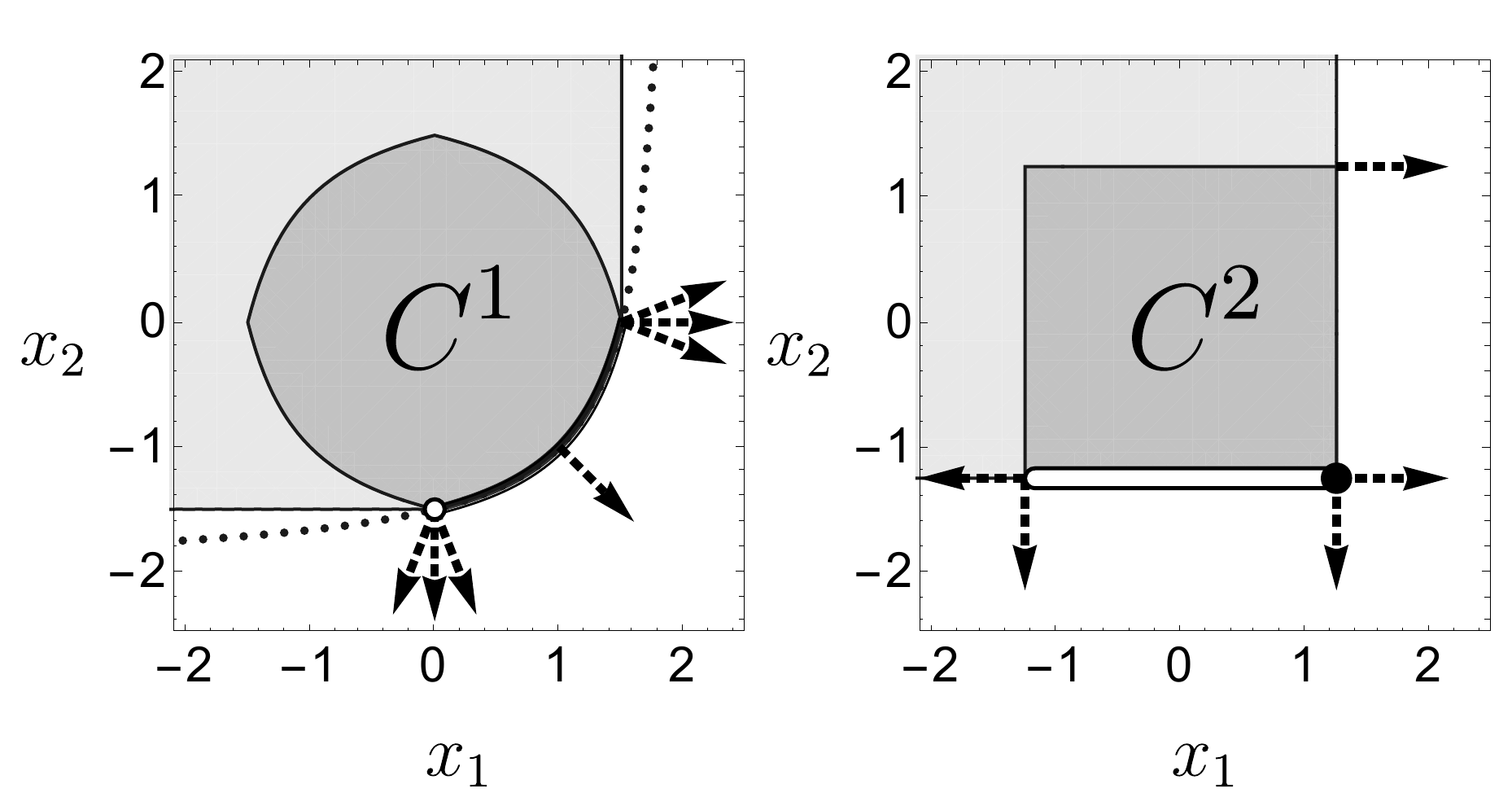}\label{normalfigb}}
\caption{An illustration of Propositions~\ref{bdprop} and \ref{newtheoref} for Examples~\ref{example2} and \ref{example2b}.}\label{normalfig}
\end{figure}

Example~\ref{example2} illustrates how the characterization of $\bd\bra{Q\bra{\mathcal{C}}}$ from Proposition~\ref{bdprop} can be turned into a piecewise description composed of a finite number of sets (e.g. $\hat{B}$, $\tilde{B}$, $C^1 \times \set{\e^1}$, etc.).	All sets associated to \eqref{relbdunion2} have simple explicit descriptions that yield trivial valid inequalities for $Q\bra{\mathcal{C}}$ (e.g. $C^1 \times \set{\e^1}$ yields $y_1\leq 1$ or equivalently $y_2\geq 0$). In contrast, the sets associated \eqref{relbdunion1} yield non-trivial valid inequalities, but do not always have clear explicit descriptions (e.g. $\tilde{B}$ yields $-x_2\leq (3/2)y_1+(5/4)(1-y_1)$, but the non-linear inequality associated to $\hat{B}$ is harder to describe). Fortunately, it is sometimes possible to directly obtain a finite piecewise description of $Q\bra{\mathcal{C}}$. The first step is to describe $Q\bra{\mathcal{C}}$ as a finite intersection of similar sets, but with known descriptions.
\begin{proposition}\label{newtheoref}Let $\mathcal{C}:=\set{C^i}_{i=1}^k\in \mathbb{C}_n$,   $\mathcal{C}^j:=\set{C^{j,i}}_{i=1}^k\in \mathbb{C}_n$ for each $j\in \sidx{m}$ and $U=\bigcap_{i=1}^k \dom\bra{\sigma_{C^i}} \setminus\set{{\bf 0}}$ or $U=\Real^n$. If
\begin{subequations}\label{parboth}
\begin{alignat}{3}
C^i&\subseteq C^{j,i} &\quad& \forall j\in \sidx{m},\,i\in \sidx{k}\label{par1}\\
\forall u\in U\quad \exists j\in\sidx{m} \text{ s.t. }\sigma_{C^i}\bra{u}&=\sigma_{C^{j,i}}\bra{u}&\quad& \forall i\in \sidx{k}
,\label{par2}
\end{alignat}
\end{subequations}
then $Q\bra{\mathcal{C}}=\bigcap_{j=1}^m Q\bra{\mathcal{C}^{j}}$.
\end{proposition}
\begin{proof}
Let $Q=\bigcap_{j=1}^m Q\bra{\mathcal{C}^{j}}$. Condition \eqref{par1} implies $Q\bra{\mathcal{C}}\subseteq Q$. For $Q\subseteq Q\bra{\mathcal{C}}$ we show
$\sigma_{ Q}\bra{u,v}\leq \sigma_{Q\bra{\mathcal{C}} } \bra{u,v}$. If $\sigma_{Q\bra{\mathcal{C}} } \bra{u,v}=\infty$ this holds trivially,  so we  assume $\bra{u,v}\in \dom\bra{\sigma_{Q\bra{\mathcal{C}}}}$. By Theorem 3.3.2 in \cite{hiriart-lemarechal-2001} we have
\begin{equation}\label{maxsup}
\sigma_{Q\bra{\mathcal{C}}}\bra{u,v}=\max\nolimits_{i=1}^k \sigma_{C^i}\bra{u}+ v\cdot e^i.
\end{equation}
Then $\dom\bra{\sigma_{Q\bra{\mathcal{C}}}}= \bra{\bigcap_{i=1}^k \dom\bra{\sigma_{C^i}}}\times \Real^k$ and $u\in\dom\bra{\sigma_{C^i}}$ for all $i\in \sidx{k}$.
If $u=0$, then $\sigma_{Q}\bra{u,v}=\sigma_{Q\bra{\mathcal{C}} } \bra{u,v}=\max_{i=1}^k v_i$.  For $u\neq 0$ let $j\in \sidx{m}$ be the index from condition \eqref{par2}. Combining this condition with \eqref{maxsup} and Theorem 3.3.2 in \cite{hiriart-lemarechal-2001}  for $Q\bra{\mathcal{C}^{j}}$ we finally have
$\sigma_{ Q}\bra{u,v}\leq \sigma_{Q\bra{\mathcal{C}^{j}}}\bra{u,v}=\max\nolimits_{i=1}^k \sigma_{C^{j,i}}\bra{u}+ v\cdot e^i=\sigma_{Q\bra{\mathcal{C}}}\bra{u,v}$. \qed
\end{proof}
The second step is to combine Proposition~\ref{newtheoref} with the known descriptions of the $Q\bra{\mathcal{C}^{j}}$. For instance, below we  combine it  with Proposition~\ref{homothetict}.
\begin{corollary}\label{generalhcoro}
Let $\mathcal{C}:=\set{C^i}_{i=1}^k\in \mathbb{C}_n$ and 	for each $j\in \sidx{m}$ let $C^{j,0}\subseteq \Real^n$ with ${\bf 0}\in C^{j,0}$,  $\set{b^{j,i}}_{i=1}^k\subseteq \mathbb{R}^n$,  $r^j\in \Real^k_+\setminus \set{0}$ and $\mathcal{C}^j:=\set{C^{j,i}}_{i=1}^k$ be such that $C^{j,i}=r^j_i C^{j,0}+b^{j,i} + C^{j,0}_\infty$ for all $i\in \sidx{k}$ and $j\in \sidx{m}$. If \eqref{parboth} holds for $\set{\mathcal{C}^j}_{j=1}^m$, then an ideal formulation for $x\in \bigcup_{i=1}^k C^i$ is given by
\begin{equation}\label{complexform}
  \gamma_{C^{j,0}}\left(x- \sum_{i=1}^k y_i b^{j,i}\right)\leq \sum_{i=1}^k r^j_i y_i \; \forall j\in \sidx{m},\quad
   \sum_{i=1}^k y_i=1,\quad y\in\set{0,1}^k.
\end{equation}
\end{corollary}

\begin{example}\label{example2b} Let $C^1$ and  $C^2$ again be the sets from Example~\ref{example1} depicted in Figures~\ref{ex1fig} and \ref{normalfigb}. To  construct an ideal formulation for $x\in C^1\cup C^2$ we divide directions $u\in \Real^2\setminus\set{0}$ for condition \eqref{par2}  into four  classes.
For each $s\in \set{-1,1}^2$ let $C^{s,1}:=\set{x\in \Real^2\,:\, \bra{2- s_1  x_1 }\bra{2- s_2  x_2 }\geq 1,\; s_j  x_j \leq 3/2\; \forall i\in \sidx{2}}$, $C^{s,2}:=\set{x\in \Real^2\,:\, s_j  x_j \leq 5/4\quad \forall i\in \sidx{2}}$   and $D_s:=\set{u\in \Real^2\,:\, s_1  x_1 \geq 0,\quad s_2  x_2 \geq 0}$.
For $s=(1,-1)$, Figure~\ref{normalfigb} depicts $C^{s,1}$ and $C^{s,2}$ in light gray and illustrates how condition \eqref{parboth} is satisfied: for each $i\in \sidx{2}$, $s\in \set{-1,1}$ and $u\in D_s$ we have $\sigma_{C^i}\bra{u}=\sigma_{C^{s,i}}\bra{u}$ and $C^i\subseteq C^{s,i}$. Finally, if we let  $C^{s,0}=C^{s,1}$, $r^s_1=1$, $b^{s,1}=(0,0)^T$,  $r^s_1=0$ and  $b^{s,1}=\bra{s_1(5/4),s_2(5/4)}^T$ we have $C^{s,i}=r^s_i C^{s,0}+b^{s,i} + C^{s,0}_\infty$ for all $i\in \sidx{2}$ and
\[\epi\bra{\gamma_{C^{s,0}}}=\set{\bra{x,y}\in \Real^3\,:\, \begin{alignedat}{3}\norm{\bra{2 y,\; s_1  x_1-s_2  x_2}}_2&\leq 4y-s_1 x_2-s_2 x_2,\\ s_j  x_j &\leq (3/2)y \quad \forall j\in \sidx{2}\end{alignedat}}.\]
Then \eqref{complexform} yields the ideal formulation of $x\in C^1\cup C^2$ given by
\begin{subequations}\label{idealofex1}
\begin{alignat}{3}
\norm{\bra{2 y_1,\; s_1  x_1-s_2  x_2}}_2&\leq 4y_1+(5/2)y_2-s_1 x_2-s_2 x_2 \quad \forall s\in \set{-1,1}^2,\\
 s_j  x_j &\leq (3/2)y_1 +(5/4) y_2\quad \forall j\in \sidx{2},\;s\in \set{-1,1}^2,\label{reducelinearineq}\\
 \quad y_1+y_2&=1,\quad y\in \set{0,1}^2.
\end{alignat}
\end{subequations}
where we used the fact that $s_1 ^2=s_2 ^2=1$ for all $s\in\set{-1,1}^2$ to simplify the nonlinear inequalities in $x$ and $y$.\qed
\end{example}

Note that the key to effectively satisfy condition \eqref{par2} was to include $s_j  x_j \leq 3/2$ in the definition of $C^{s,1}$. Indeed, as can be glimpsed from Figure~\ref{normalfigb} if we omitted these constraints for $s=(1,-1)$, we would have $\sigma_{C^1}\bra{0,-1}=3/2<2=\sigma_{C^{s,1}}\bra{0,-1}$. Another way to understand the need for these inequalities is by noting that for $s=(1,-1)$ they ensure that  $N_{C^{s,1}}\bra{\tilde{x}^1}\cap N_{C^{s,2}}\bra{x^2}\neq \emptyset$ for $\tilde{x}^1:=\bra{0,-3/2}$ and all $x^2\in B^2:=\set{x\in C^2\,:\,x_2=-5/4}$ (cf. white boundary subsets depicted in Figure~\ref{normalfigb} and discussed in Example~\ref{example2}). This last observation can be useful to construct families $\set{\mathcal{C}^j}_{j=1}^m$ that satisfy condition \eqref{par2} (and verify that they do satisfy it) so we formalize it in Corollary~\ref{alternativetheo} of Section~\ref{necandsufsec}. However, we first showcase some important applications where    \eqref{par2} can be easily verified.

\section{Applications of Proposition~\ref{newtheoref}}\label{applicasec}

While Proposition~\ref{newtheoref} and Corollary~\ref{generalhcoro} are simple, together with Proposition~\ref{generalorthprop} they can recover and generalize all known results from the literature.
\subsection{Unions of Polyhedra}\label{polysec}

The first result that Corollary~\ref{generalhcoro} can generalize is the following class of formulations introduced by Balas, Blair and Jeroslow \cite{balas88,blair90,jeroslow88}.

\begin{definition}
For any $A\in \mathbb{R}^{m\times n}$ and $B\subseteq \sidx{m}$ let $A_B\in\mathbb{R}^{\abs{B}\times n}$ be the sub-matrix of $A$ composed of the rows indexed by $B$.
For a fixed $A\in \mathbb{R}^{m\times n}$ let $\mathcal{B}=\set{B\subseteq \sidx{m}\,:\, \abs{B}=\rank(A), \quad \rank\bra{A_B}   =\rank(A)}$, and for any $B\in \mathcal{B}$ and $b\in \Real^m$ let $P\bra{B,b}:=\set{x\in \Real^n\,:\, A_B x\leq b_B}$ and $\bar{x}\bra{B,b}\in \Real^n$ be an arbitrary solution of $A_B x=b_B$.
\end{definition}

\begin{theorem}[Theorem 2 in \cite{blair90}]\label{blairprop} Let $A\in \mathbb{R}^{m\times n}$ and for each $i\in \sidx{k}$ let $b^i\in \Real^m$ and $P^i=\set{x\in \Real^m\,:\,Ax\leq b^i}$. If
\[\forall B\in \mathcal{B}\quad \bra{\bar{x}\bra{B,b^i}\in P^i\quad \forall i\in \sidx{k}} \quad\vee \quad\bra{\bar{x}\bra{B,b^i}\notin P^i\quad \forall i\in \sidx{k}},\]
then an ideal formulation of $x\in \bigcup_{i=1}^k P^i$ is given by
 \begin{subequations}\label{blairform}
\begin{alignat}{3}
Ax \leq \sum\nolimits_{i=1}^k b^i y_i,\quad \sum\nolimits_{i=1}^k y_i=1,\quad y\in \set{0,1}^k.
\end{alignat}
\end{subequations}
\end{theorem}
Corollary~\ref{generalhcoro} generalizes Theorem~\ref{blairprop} as follows.
\begin{corollary}\label{blairpropplus}Let $A\in \mathbb{R}^{m\times n}$ and for each $i\in \sidx{k}$ let $b^i\in \Real^m$ and $P^i=\set{x\in \Real^m\,:\,Ax\leq b^i}$. If for all $c\in \Real^n$ there exist $B\in \mathcal{B}$ such that
\begin{equation}\label{blairproppluscondition}
\max\set{c\cdot x\,:\, x\in P\bra{B,b^i}}=\max\set{c\cdot x\,:\,x\in P^i} \quad \forall i\in \sidx{k},
\end{equation}
then \eqref{blairform} is an ideal formulation of $x\in \bigcup_{i=1}^k P^i$ .
\end{corollary}
\begin{proof} For all $B\in \mathcal{B}$ let $\mathcal{C}^B:=\set{C^{B,i}}_{i=1}^k$ be such that $C^{B,i}=P\bra{B,b^i}=P\bra{B,{\bf 0}}+\bar{x}\bra{B,b^i}$ for all $i\in \sidx{k}$. Condition \eqref{par1} is trivially satisfied and condition \eqref{par2} is satisfied by the corollary's assumption. The result follows from Corollary~\ref{generalhcoro} by noting
that $\sidx{m}=\bigcup_{B\in \mathcal{B}} B$, that because $\epi\bra{\gamma_{P\bra{B,{\bf 0}}}}=\set{\bra{x,y}\in \Real^{n+1}\,:\, y\geq 0,\quad  A_B x\leq 0}$ we have  $\bra{x,y}\in Q\bra{\mathcal{C}^B}$ if and only if
\begin{alignat*}{3}
A_B\bra{x-\sum\nolimits_{i=1}^k \bar{x}\bra{B,b^i} y_i}=A_B x - \sum\nolimits_{i=1}^k  b^i_B y_i &\leq 0 &\quad& \\
\sum\nolimits_{i=1}^k y_i=1,\quad y_i&\geq 0 &\quad& \forall i\in \sidx{k}.\quad\quad\quad\qed
\end{alignat*}

\end{proof}
The sufficient condition of Theorem~\ref{blairprop} implies that of Corollary~\ref{blairpropplus}, but the following example adapted from \cite{Mixed-Integer-Linear-Programming-Formulation-Techniques} shows that the converse may not hold.
\begin{example}Consider \[A=\left[\begin{array}{rrr} 1 &0 &1\\ -1 & 0 &1 \\ 0 &1 &1\\ 0 &-1&\phantom{-} 1\end{array}\right],\quad b^1=\begin{pmatrix} 1\\1\\2\\2\end{pmatrix},\quad
	b^2=\begin{pmatrix} 2\\2\\1\\1\end{pmatrix}.\]
We can check that $B_1:=\set{1,2,3}\in \mathcal{B}$, $\bar{x}\bra{B_1,b^1}=\bra{0,1,1}\in P^1$ and $\bar{x}\bra{B_1,b^2}=\bra{0,-1,2}\notin P^2$. Furthermore, \[\max\set{x_3\,:\, x\in P\bra{B_1,b^2}}=2>1=\max\set{x_3\,:\,x\in P^2}.\]
Then, neither Theorem~\ref{blairprop} nor Corollary~\ref{blairpropplus} are applicable and indeed formulation \eqref{blairform} for these matrix/vectors is not ideal ($x=\bra{0,0,3/2}$ and $y=\bra{1/2,1/2}$ is an extreme point of its LP relaxation). However, if we augment $A$, $b^1$ and $b^2$ with the redundant inequality $x_3\leq 1$ (i.e. let the fifth row of $A$ be $\bra{0,0,1}$ and $b^1_5=b^2_5=1$) we have that  $B_2:=\set{1,2,5}\in \mathcal{B}$ and
\[\max\set{x_3\,:\, x\in P\bra{B_2,b^2}}=1=\max\set{x_3\,:\,x\in P^2}.\]
Moreover, with this additional inequality/row we have that for any $u\in \Real^3 $ condition \eqref{blairproppluscondition} either holds trivially (i.e. with $+\infty$ on both sides) or for a basis of the form $B=\set{i,j,5}$ for $i,j\in \sidx{4}$. Hence, Corollary~\ref{blairpropplus} shows that \eqref{blairform} for this augmented matrix/vectors does yield an ideal formulation for $x\in P^1\cup P^2$. In contrast, we still have $\bar{x}\bra{B_1,b^1}\in P^1$ and $\bar{x}\bra{B_1,b^2}\notin P^2$ for the augmented matrix/vectors so Theorem~\ref{blairprop} cannot be used to prove that this formulation is ideal.\qed
\end{example}

Theorem~\ref{blairprop} and Corollary~\ref{blairpropplus} are based on exploiting a common tangent structure of the  $P^i$. This can also be useful to (partially) satisfy  condition  \eqref{parboth} for non-polyhedral sets so we give one formalization of the approach.

\begin{lemma}\label{commontangent1} Let  $\mathcal{C}:=\set{C^i}_{i=1}^k\in \mathbb{C}_n$, $\set{x^{j,i}}_{j=1}^m\subseteq \bd\bra{C^i}$ for all $i\in \sidx{k}$,  $\mathcal{C}^j:=\set{C^{j,i}}_{i=1}^k\in \mathbb{C}_n$ for all $j\in \sidx{m}$ and $C^{j,0}\subseteq \Real^n$ be a closed convex cone for all $j\in \sidx{m}$. If  $C^{j,i}=T_{C^i}\bra{x^{j,i}}=x^{j,i}+C^{j,0}$ for all $i\in \sidx{k}$ and $j\in \sidx{m}$, then $\set{\mathcal{C}^j}_{j=1}^m$ satisfies  \eqref{parboth} for $U=\bigcup_{j=1}^m \bra{C^{j,0}}^*$.
\end{lemma}
\begin{proof}Direct from $\sigma_{C}\bra{u}=\sigma_{T_C\bra{x}}\bra{u}$ for all $x\in \bd\bra{C}$ and $u\in T_C\bra{x}^*$.\qed
\end{proof}
\subsection{Common tangent structure through Minkowski sum}

Example~\ref{example2b} uses a ``nearly-homothetic'' variant of ``conic'' tangents of Lemma~\ref{commontangent1}. For instance, as illustrated in Figure~\ref{normalfigb} for $s=(1,-1)$ and $\bar{x}=(5/4,-5/4)$,  we have that $C^{s,2}=\bar{x}+C^{s,0}_\infty$ is the  cone tangent to $C^2$ at $\bar{x}$, but no translation of $C^{s,0}_\infty$ is tangent to $C^{1}$ at some $x\in \bd\bra{C^1}$. However, $C^{s,1}=C^1+C^{s,0}_\infty$ serves the same role as the translation of $C^{s,0}_\infty$ through the following property.
\begin{lemma}\label{commontangent2} Let $C\subseteq \Real^n$ be a closed convex set and $K\subseteq \Real^n$ be a closed convex cone. Then $\sigma_C\bra{u}=\sigma_{C+K}\bra{u}$ for all $u\in K^*$.
\end{lemma}
\begin{proof}
By Theorem C.3.3.2 in \cite{hiriart-lemarechal-2001} $\sigma_{C+K}\bra{u}=\sigma_{C}\bra{u}+\sigma_{K}\bra{u}$ and $\sigma_{K}\bra{u}=0$ for all $u\in K^*$. \qed
\end{proof}

The following example further illustrates this approach to guide the construction of $\set{\mathcal{C}^j}_{j=1}^m $ to use Corollary~\ref{generalhcoro}  for non-polyhedral sets with $k,n>2$. It also illustrates how redundancy in $\set{\mathcal{C}^j}_{j=1}^m $ can simplify verification of   \eqref{parboth}.

\begin{example}\label{generalexamplecone}
Let $b=\sqrt{2}-1$, $r\in \Real^k_+$, $\set{s^i}_{i=1}^k \subseteq \set{0,1}^2$,  $\set{\bra{p_0^i,p^i}}_{i=1}^k\subseteq \Real^{n+1}$, $G^i:=\set{\bra{x_0,x}\in \Real^{n+1}\,:\, \begin{alignedat}{3}\norm{\bra{x,s^i_l }}_2 &\leq  s^i_l \bra{\sqrt{2}-1}+1 + (-1)^l   x_0 \quad \forall l\in \sidx{2}\end{alignedat}}$ and $C^i=\bra{p_0^i,p^i }+r_i G^i$ for all $i\in \sidx{k}$.
Family $\mathcal{C}:=\set{C^i}_{i=1}^k$ is depicted in Figure~\ref{gcfig1} for $k=2$, $n=1$, $r=(1,1)$, $s^1=\bra{1,0}$, $s^1=\bra{0,1}$ and $\bra{p_0^i,p^i}=\bra{0,0}$ for all $i\in \sidx{2}$.
Let $\set{\mathcal{C}^j}_{j=1}^2$ be such that for each $j\in \sidx{2}$
\[C^{j,0}:=\set{\bra{x_0,x}\in \Real^{n+1}\,:\,\norm{\bra{x,1}}_2 \leq  \sqrt{2}+ (-1)^j x_0,\; \norm{x}_2 \leq  1 + (-1)^j x_0}\]
and $C^{j,i}=\bra{p_0^i-(-1)^j (1-s^i_j)r_i,p^i}+r_i s^i_j C^{j,0}$ for all $i\in \sidx{k}$. Then $C^{2,0}_\infty=-C^{1,0}_\infty=\set{\bra{x_0,x}\in \Real^{n+1}\,:\, \norm{x}_2 \leq  x_0}$ and for all $j\in \sidx{2}$ and $i\in \sidx{k}$ we have $C^{j,i}=T_{C^i}\bra{p^i_0-(-1)^j,p^i}$ if $s^i_j=0$ and $C^{j,i}=C^i+C^{j,0}_\infty$ if $s^i_j=1$. Then $\set{\mathcal{C}^j}_{j=1}^2$ satisfies  \eqref{parboth} for $U=\bra{C^1_\infty}^*\cup\bra{C^2_\infty}^*$. Furthermore, for each $j\in \sidx{2}$
\begin{alignat*}{3}\epi\bra{\gamma_{C^{j,0}}}&=\set{\bra{x_0,x,y}\in  \Real^{n+2}\,:\, \begin{alignedat}{3}\norm{\bra{x,y }}_2 &\leq  y (b +1) + (-1)^j  x_0,\\ \norm{x}_2 &\leq  y+ (-1)^j  x_0\end{alignedat}}
 \end{alignat*}
For  all $i\in \sidx{k}$ and $j\in \sidx{2}$ let $q^{i}_j=r_i s_j^i \sqrt{2}-(-1)^j p_0^j+(1-s_j^i)r_i$ and $h^{i}_j=r_i s_j^i -(-1)^j p_0^j+(1-s_j^i)r_i$. Because $C^i=C^{1,i}\cap C^{2,i}$ formulation \eqref{complexform} for yields the valid formulation of  $\bra{x_0,x}\in \bigcup_{i=1}^k C^i$ given by
\begin{subequations}\label{generalconeex}
\begin{alignat}{3}
\norm{\bra{x-\sum\nolimits_{i=1}^k p^i y_i, \sum\nolimits_{i=1}^k r_i s^i_j y_i} }_2 &\leq  \sum\nolimits_{i=1}^k q^{i}_j y_i+(-1)^j  x_0, \; \forall j\in \sidx{2} \\
 \norm{x-\sum\nolimits_{i=1}^k p^i y_i}_2 &\leq  \sum\nolimits_{i=1}^k h^i_j y_i +(-1)^j  x_0\label{generalconeexreds}, \; \forall j\in \sidx{2} \\
\sum\nolimits_{i=1}^k y_i&=1,\quad y\in \set{0,1}^k.
\end{alignat}
\end{subequations}
  We have  $\bra{C^1_\infty}^*\cup\bra{C^2_\infty}^*=C^2_\infty\cup C^1_\infty$  strictly contained in $\bigcap_{i=1}^k \dom\bra{\sigma_{C^i}} \setminus\set{0}$, so Corollary~\ref{generalhcoro} does not imply idealness of \eqref{generalconeex}. To check that it is indeed ideal let $\mathcal{C}^3$ be such that  $C^{3,0}:=\set{\bra{x_0,x}\in \Real^{n+1}\,:\,\norm{x}_2 \leq   1+  x_0,\quad \norm{x}_2 \leq   1-  x_0}$ and $C^{3,i}=\bra{p_0^i,p^i}+r_i C^{3,0}$ for all $i\in \sidx{k}$.
  Then  for all $i\in \sidx{k}$ and $u\in \Real^n \setminus \bra{C^2_\infty\cup C^1_\infty}$  we have  $C^i\subseteq C^{3,i}$ and $\sigma_{C^{3,i}}\bra{u}=\sigma_{C^i}\bra{u}$. Finally, we have
$
\epi\bra{\gamma_{C^{3}}}=\set{\bra{x_0,x,y}\in  \Real^{n+2}\,:\, \norm{x}_2 \leq   y+  x_0,\quad \norm{x}_2 \leq   y-  x_0}$.

We can now use Corollary~\ref{generalhcoro} for $\set{\mathcal{C}^j}_{j=1}^3$ to construct an ideal formulation of $\bra{x_0,x}\in \bigcup_{i=1}^k C^i$ that corresponds to \eqref{generalconeex} plus the inequalities associated to $\mathcal{C}^3$. However, these additional inequalities are precisely \eqref{generalconeexreds}. \qed
\end{example}

\begin{figure}[htbp]
\centering \includegraphics[scale=.35]{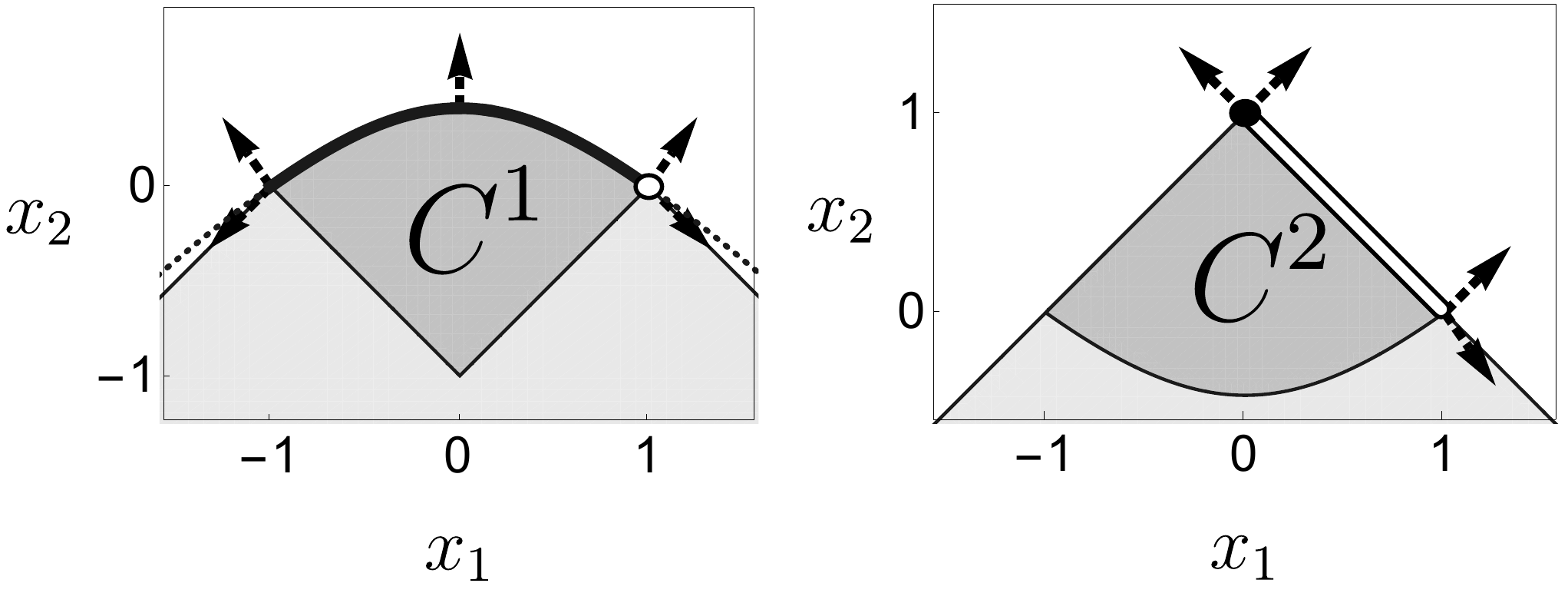}
\caption{Sets from Example~\ref{generalexamplecone}.}\label{gcfig1}
\end{figure}

Note that in Example~\ref{example2b} the approach based on Lemma~\ref{commontangent2} of adding $C^{s,0}_\infty$ to $C^1$ in the definition of $C^{s,1}$ (or equivalently including $s_j  x_j \leq 3/2$ in the definition of $C^{s,1}$) was enough to yield an ideal formulation and to verify property \eqref{parboth}. In contrast, in Example~\ref{generalexamplecone} this approach was enough to yield an ideal formulation, but not to verify the property. We further discuss this in Section~\ref{necandsufsec} where the boundary subsets highlighted in white in Figure~\ref{gcfig1} will play a similar role to those in Figure~\ref{normalfigb}.

\subsection{Constraints from power systems applications}

A very clever technique to extend the applicability of Theorem~\ref{isotone} was introduced by \cite{bestuzheva2016convex} in the context of power systems. The following example illustrates how this technique relates to the use of Lemma~\ref{commontangent2} and Corollary~\ref{generalhcoro}.
\begin{example}\label{quadrsticsin} Let $C^1:=[-1,1]\times \set{0}$,
$K\bra{l,u}:=\set{x\in [l,u]\times [0,1]\,:\, x_1^2\leq x_2}$ and $C^2:=K\bra{-1,1}$.  We have that $\{C^i\}_{i=1}^2$  does not satisfy the assumptions of Theorem~\ref{isotone}. However, \cite{bestuzheva2016convex} notes that if $\tilde{C}^2=K\bra{-1,0}$ or $\tilde{C}^2=K\bra{0,1}$, then (after a rotation) $\{C^1,\tilde{C}^2\}$ does satisfy the assumptions. Hence,  Theorem~\ref{isotone} can characterize $Q^1:=\conv\bra{\bra{C^1\times \set{\e^1}} \cup \bra{K\bra{-1,0}\times \set{\e^2}}  }$ and $Q^2:=\conv\bra{\bra{C^1\times \set{\e^1}}  \cup \bra{K\bra{0,1}\times \set{\e^2}}  }$.
Then \cite{bestuzheva2016convex} further notes  $\conv\bra{\bra{C^1\times \set{\e^1}} \cup \bra{K\bra{-1,1}\times \set{\e^2}} }=\conv\bra{Q^1\cup Q^2}$ and using the construction of the $Q^i$ from Theorem~\ref{isotone} shows that $Q^1\cup Q^2$  convex   and
\[
 Q^1\cup Q^2= \set{\bra{x,y}\in \Real^4\,:\,\begin{alignedat}{6}
x_1-y_1&\leq \sqrt{x_2 y_2},\quad&  0 \leq x_2&\leq y_2, \quad& y_1+y_2&=1, \\-x_1-y_1&\leq \sqrt{x_2 y_2},\quad& -1\leq x_1&\leq 1,  \quad& y_1,y_2&\geq 0
\end{alignedat}}.
\]
To instead construct a formulation using Corollary~\ref{generalhcoro} let $\set{\mathcal{C}^j}_{j=1}^2$ be such that  \blue{$C^{j,0}=C^2+\{x\in \Real^2\,:\, \bra{-1}^j x_1\leq 0,\quad x_2=0\}$},    $C^{j,1}=\bra{-1}^j \e^1+C^{j,0}_\infty=\{x\in \Real^2\,:\, \bra{-1}^j x_1\leq 1,\quad x_2=0\}$ and $C^{j,2}=C^{j,0}$ for each $j\in \sidx{2}$. Then $C^{j,0}_\infty=\{x\in \Real^2\,:\, \bra{-1}^j x_1\leq 0,\quad x_2=0\}$, $C^{j,1}=T_{C^1}(\bra{-1}^j,0)$ and $C^{j,2}=C^2+C^{j,0}_\infty$ for all $j\in \sidx{2}$. Then $\set{\mathcal{C}^j}_{j=1}^2$ satisfies  \eqref{parboth} for $U=\Real^2$\blue{, so we can use Corollary~\ref{generalhcoro}. To  obtain an explicit  algebraic description of the resultant formulation, first note that for each $j\in \sidx{2}$ we have  $C^{j,0}=\{x\in \Real^2\,:\, g_j(x_1)\leq x_2,\; x_2\leq 1\}$, where
\[g_j(x_1)=((\bra{-1}^j x_1)^+)^2=\begin{cases}x_1^2& \text{if }\bra{-1}^j x_1 \geq 0 \\ 0 & \text{if }\bra{-1}^j x_1 \leq 0\end{cases}.\]
This construction is illustrated in Figure~\ref{fee1} for $j=1$, where $C^{1,0}$ is depicted  in gray, the graph of $x_1^2$ is depicted by the solid black curve and the graph of $g_1(x_1)$ is depicted by the dotted black curve. Figure~\ref{fee1} illustrates the convexity of $g_j(x_1)$, which we can use to conclude that }  for each $j\in \sidx{2}$ we have
$\epi\bra{\gamma_{C^{j,0}}}=\set{\bra{x,y}\in  \Real^{3}\,:\, \begin{alignedat}{3}((\bra{-1}^j x_1)^+)^2&\leq y\cdot x_2, \;x_2&\leq y \end{alignedat}}$. \blue{Finally, the ideal formulation for $x\in C^1\cup C^2$ from Corollary~\ref{generalhcoro}}  is given by
 \begin{equation}\label{notbasicform}
\bra{((-1)^j x_1-y_1)^+}^2\leq y_2\cdot x_2\,  \forall j\in \sidx{2},\,  x_2\leq y_2, \, y_1+y_2=1, \, y\in \set{0,1}^2.
\end{equation}
The continuous relaxation of this formulation is identical to $Q^1\cup Q^2$.\qed
\end{example}
\begin{figure}[htpb]
\centering
\subfigure[\blue{Construction for Example~\ref{quadrsticsin}.}]{\includegraphics[scale=0.29]{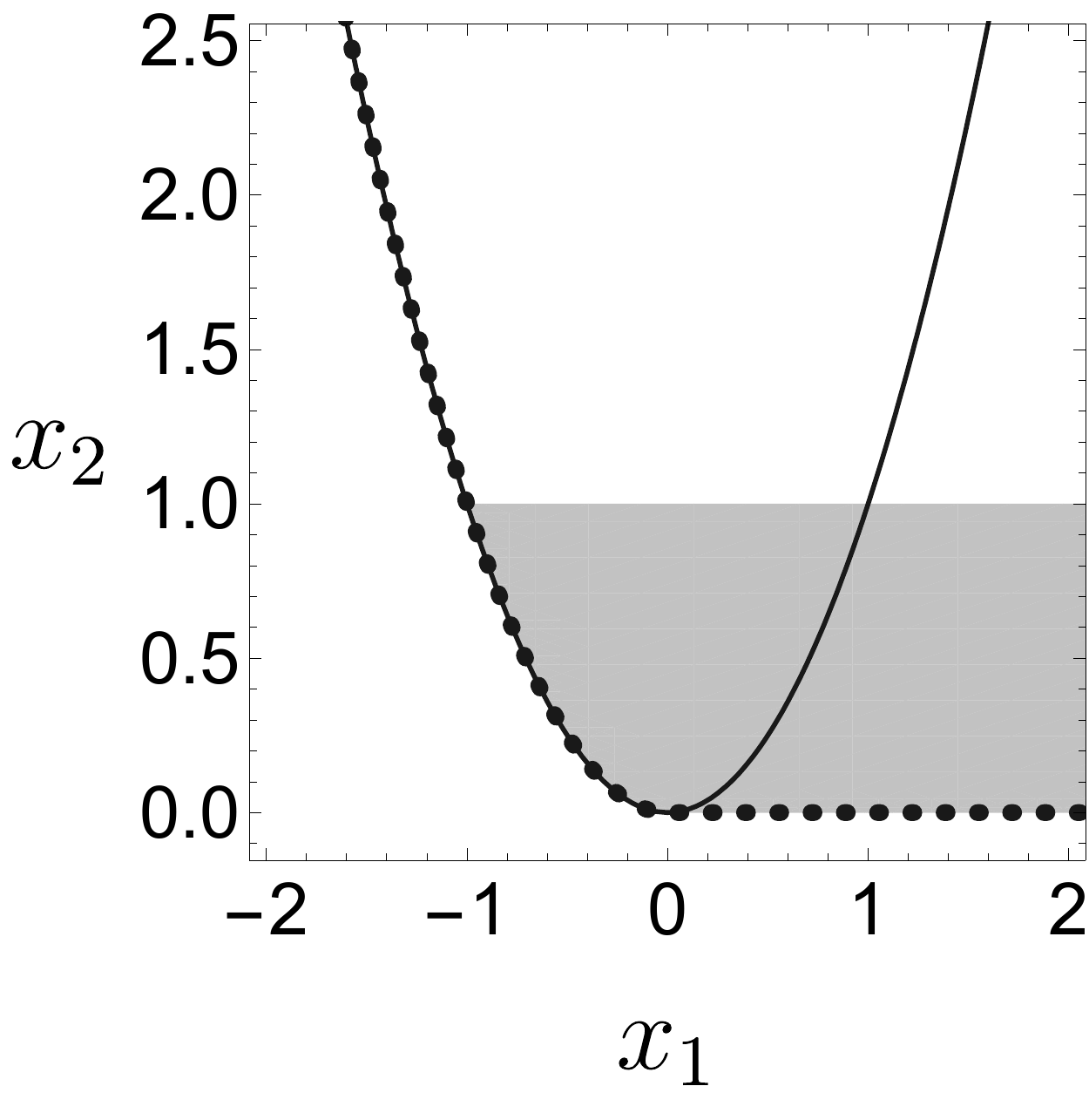}\label{fee1}}\quad
\subfigure[\blue{First attempt for Example~\ref{quadrsticsin2}.}]{\includegraphics[scale=0.29]{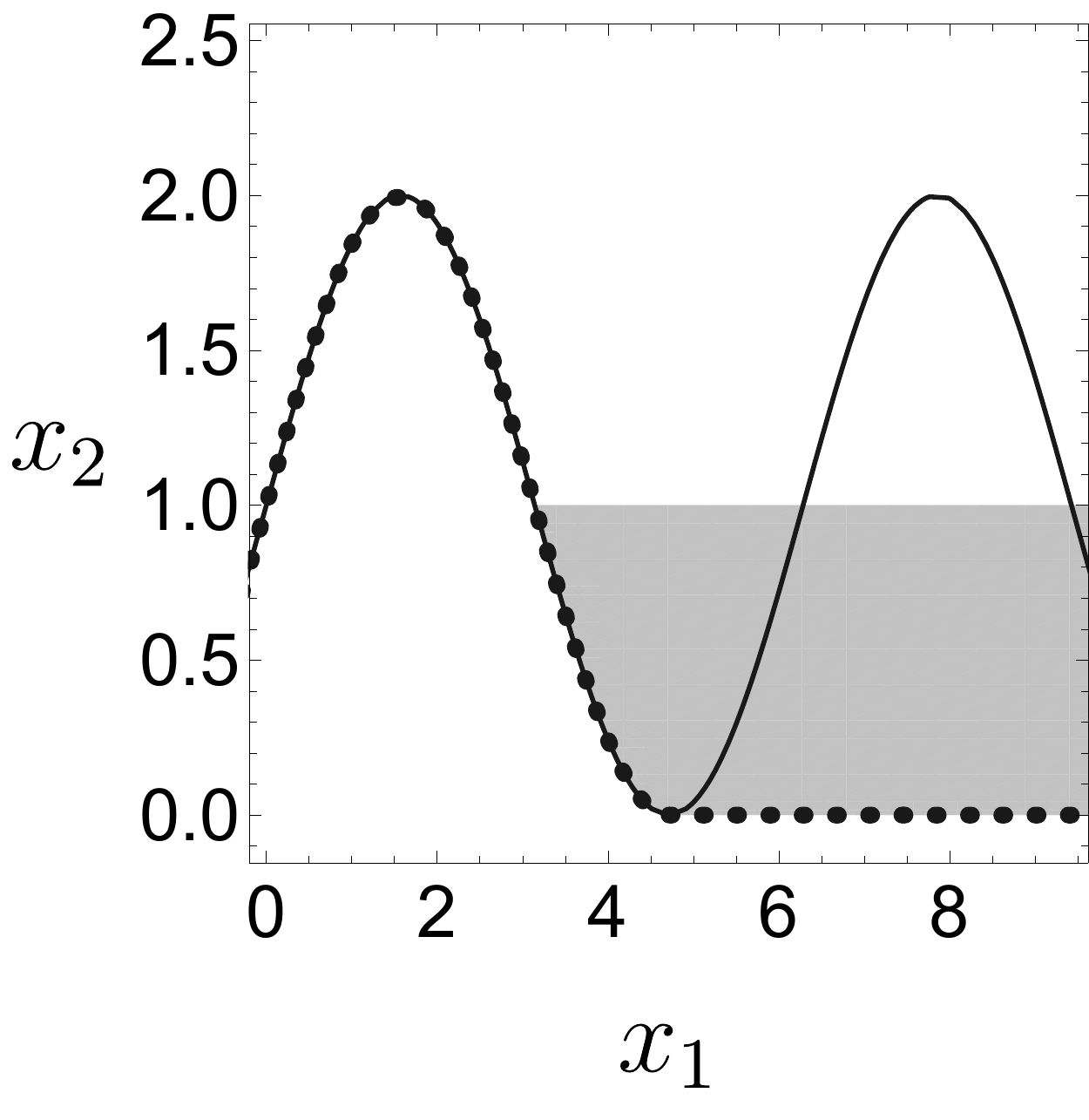}\label{fee2}}\quad
\subfigure[\blue{Final construction for Example~\ref{quadrsticsin2}.}]{\includegraphics[scale=0.29]{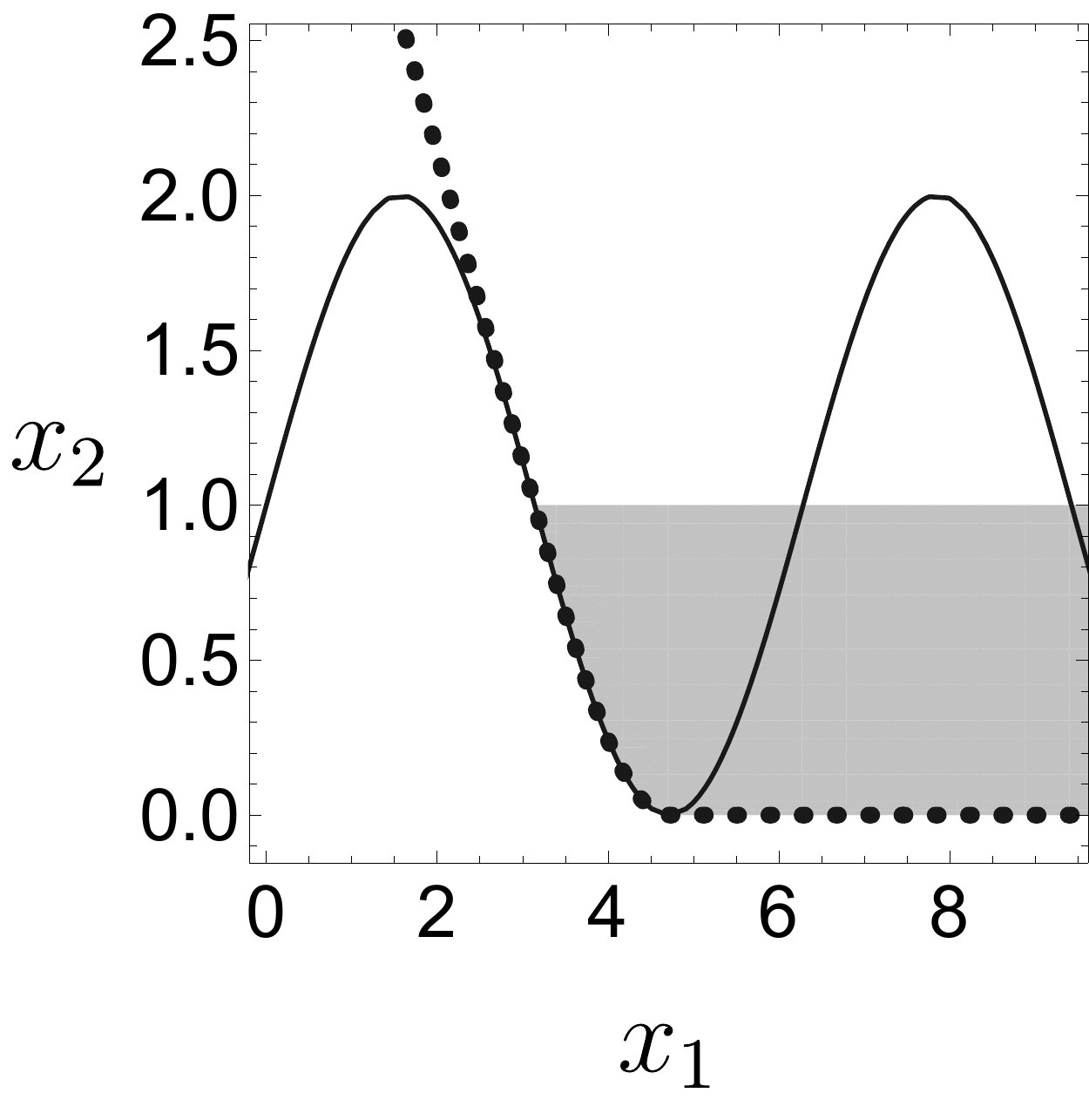}\label{fee3}}
\caption{\blue{Using $\bra{\cdot}^+$ and $\max$ to describe Minkowski sums.}}\label{fee}
\end{figure}
The quadratic set considered in \cite{bestuzheva2016convex} was an approximation of a trigonometric set \cite{bestuzheva2016convex,hijazi2013convex}. The following example shows that the Lemma~\ref{commontangent2} and Corollary~\ref{generalhcoro} can also be applied directly to such sets.

\begin{example}\label{quadrsticsin2}  Let $C^2:=\set{x\in [\pi,2\pi]\times [0,1]\,:\, \sin\bra{x_1}+1\leq x_2}$, $C^1:=[\pi,2\pi]\times \set{0}$, \blue{$\set{\mathcal{C}^j}_{j=1}^2$ be such that  $C^{j,0}=C^2+\{x\in \Real^2\,:\, \bra{-1}^j x_1\leq 0,\quad x_2=0\}$,    $C^{1,1}=\pi\e^1+C^{1,0}_\infty$,  $C^{2,1}=2\pi\e^1+C^{2,0}_\infty$, and $C^{j,2}=C^{j,0}$ for each $j\in \sidx{2}$.  Similarly to Example~\ref{quadrsticsin}, we can check that $\set{\mathcal{C}^j}_{j=1}^2$ satisfies  \eqref{parboth} for $U=\Real^2$ and we can use Corollary~\ref{generalhcoro}. To  obtain an explicit  algebraic description of the resultant formulation  we can again write  $C^{1,0}=\{x\in \Real^2\,:\, g_1(x_1)\leq x_2,\; x_2\leq 1\}$, where now
  \begin{alignat*}{3}
  g_1(x_1)-1=\sin\bra{\max\set{3\pi-x_1,3\pi/2}}&=\sin\bra{3\pi+\max\set{-x_1,-3\pi/2}}\\&=\sin\bra{3\pi-\min\set{x_1,3\pi/2}}\\&=\sin\bra{\min\set{x_1,3\pi/2}}\\&=\begin{cases}\sin\bra{x_1}& \text{if } x_1 \leq 3 \pi/2 \\ \sin\bra{3 \pi/2} & \text{if } x_1 \geq 3 \pi/2\end{cases}.
\end{alignat*}
This construction is illustrated in Figure~\ref{fee2}, where $C^{1,0}$ is depicted  in gray, the graph of $\sin\bra{x_1}+1$ is depicted by the solid black curve and the graph of $g_1(x_1)$ is depicted by the dotted black curve. Figure~\ref{fee2} shows that while $g_1(x_1)$ can be used to describe $C^{1,0}$, $g_1(x_1)$ is not convex. However, we can check that  $h_1\bra{x_1}=\max\set{g_1\bra{x_1}, \pi - x_1 }$ is convex and we can replace $g_1$ by $h_1$ in the algebraic description of $C^{1,0}$. This  is illustrated in Figure~\ref{fee3}, where again $C^{1,0}$ is depicted  in gray, the graph of $\sin\bra{x_1}+1$ is depicted by the solid black curve and now the the dotted black curve depicts the graph of $h_1(x_1)$. Using a similar reasoning for $C^{2,0}$, we can check that for both $j\in \sidx{2}$ we have} $C^{j,0}=\set{x\in \Real^2\,:\, f_i\bra{x}\leq 0,\quad \bra{-1}^j x_1\leq \bra{-1}^j j\cdot \pi ,\quad 0\leq x_2\leq 1}$, \blue{where}
$f_1\bra{x}=\blue{1-x_2+h_1\bra{x_1}}=1-x_2+\max\set{\sin\bra{\max\set{-x_1+3 \pi , 3 \pi/2} }, \pi - x_1 }$
and $f_2\bra{x_1}=1-x_2+\max\set{\sin\bra{\max\set{x_1 , 3 \pi/2} }, x_1 - 2 \pi}$. Finally, \blue{we can check that}
\[
 (\cl \tilde{f_i})(x,y)=\begin{cases}y f_i\bra{x/y}& y>0\\ y-x_2+(\bra{-1}^j x_1)^+ & y=0 \\\infty &y<0\end{cases}\]
and $\epi\bra{\gamma_{C^{j,0}}}=\{\bra{x,y}\in  \Real^{3}\,:\, (\cl \tilde{f_i})(x,y)\leq 0,\quad  \bra{-1}^j x_1\leq  \bra{-1}^j j\cdot \pi \cdot y,$ $0\leq x_2\leq y\}$.
Then  we can then use Corollary~\ref{generalhcoro} to obtain the ideal formulation of $x\in C^1\cup C^2$ given by
 \begin{alignat*}{3}
(\cl \tilde{f_1})\bra{x_1-\pi y_2,x_2,y_1}\leq 0,\quad  (\cl \tilde{f_2})\bra{x_1-2 \pi y_2,x_2,y_1}&\leq 0,\\ \pi\leq x_1\leq 2\pi,\quad 0 \leq x_2\leq y_2, \quad y_1+y_2=1, \quad y&\in \set{0,1}^2.\hspace{0.5in}\qed
\end{alignat*}
\end{example}
\subsection{Generalization of Theorem~\ref{isotone}}\label{isotonegensec}

Lemma~\ref{commontangent2} and Proposition~\ref{newtheoref} can also be used to generalize Theorem~\ref{isotone} by combining them with Proposition~\ref{generalorthprop}.

\begin{theorem}\label{isotonegeneral}
Let $\set{v^j}_{j=1}^n\subseteq \Real^n$ be an orthonormal basis of $\Real^n$ and for each $s\in \set{-1,0,1}^n$ let  $K^s=\cone\bra{\set{s_j v^j}_{j=1}^n}$.  In addition, let $\set{G^i}_{i=1}^k$ be closed convex sets in $\Real^n$ such that ${\bf 0} \in G^i$ for all $i\in \sidx{k}$, $\set{s^i}_{i=1}^k\subseteq \set{-1,1}^n$ and $D^i=G^i\cap K^{s^i}$ for each $i\in \sidx{k}$ be such that
\begin{enumerate}
	\item  $\bra{D^i-K^{s^i}}\cap K^{s^i}=D^i$ and is compact for all $i\in \sidx{k}$,
	\item for all $t\in \set{-1,1}^n$ there exist disjoint sets $\set{J^t_i}_{i=1}^k$ such that $J^t_i\subseteq \sidx{n}$ for all $i\in \sidx{k}$ and
	$D^i+K^t=D^i\cap \spann\bra{\set{v^j}_{j\in J^t_i}} +K^t \quad \forall i\in \sidx{k}$.\footnote{with $\spann\bra{\emptyset}=\set{\bf 0}$.}
\end{enumerate}
Finally, let $\set{b^i}_{i=1}^k$ and for all $i,l\in \sidx{k}$ and $j\in \sidx{n}$ let $C^i:=D^i+b^i$, $v^{i,j}:=s^i_j v^j$, $\overline{b}^{i,l}_j=\max\set{v^{i,j}\cdot x\,:\, x\in C^l}$ for $i\neq l$, $\overline{b}^{i,i}_j=v^{i,j}\cdot b^i$, $L^{i}_j:=\min\set{ v^{j}\cdot x\,:\, x\in C^i}$ and $U^{i}_j:=\max\set{ v^{j}\cdot x\,:\, x\in C^i}$. Then an ideal formulation for $x\in \bigcup_{i=1}^k C^i$ is given by
\begin{subequations}\label{isotonegeneralform}
\begin{alignat}{3}
\gamma_{G^i}\bra{\sum\nolimits_{j=1}^n v^{i,j} \bra{v^{i,j}\cdot x-\sum\nolimits_{l=1}^n \overline{b}^{i,l}_j y_l }^{\hspace{-0.05in}\text{\relsize{3}{$+$}}\hspace{0.05in}}}&\leq y_i&\quad&\forall i\in \sidx{k}\label{isotonegeneralformgau}\\
\sum\nolimits_{i=1}^k L^{i}_j y_i \leq v^j\cdot x\leq \sum\nolimits_{i=1}^k  &U^{i}_j y_i&\quad&\forall j\in \sidx{n}\label{isotonegeneralformlin}\\
\sum\nolimits_{i=1}^k y_i=1, \quad y\in \{0&,1\}^k.
 \end{alignat}
 \end{subequations}
 In particular, if $G^i=\set{x\in H^i\,:\, v^{i,j}\cdot \bra{x+b^i}\leq \overline{b}^{i,i}_j,\quad \forall j\in \sidx{n}}$ for a closed convex set $H^i\subseteq \Real^n$, then  we can replace $\gamma_{G^i}$ by $\gamma_{H^i}$ in \eqref{isotonegeneralformgau}.
 \end{theorem}
\begin{proof}
For each $t\in \set{-1,1}^n$ and $i\in \sidx{k}$ let $C^{s,i}:=C^i+K^t=D^i+K^t+b^i$. We trivially have $C^i\subseteq C^{t,i}$ for all $t\in \set{-1,1}^n$ and $i\in \sidx{k}$. Furthermore, for all $t\in \set{-1,1}^n$  and $u\in - K^t$ we have $\sigma_{C^{s,i}}\bra{u}=\sigma_{C^i}\bra{u}$ for all $i\in \sidx{k}$ because $\bra{K^t}^*=-K^t$ and $C^i$ is compact. Then, because $\Real^n=\bigcup_{t\in \set{-1,1}^n} K^t$ and Proposition~\ref{newtheoref} we have $Q\bra{\mathcal{C}}=\bigcap_{t\in \set{-1,1}^n} Q\bra{\mathcal{C}^{t}}$ for $\mathcal{C}^t:=\set{C^{t,i}}_{i=1}^k$. Noting that  $C^{t,i}:=C^i+K^t=D^i+K^t+b^i=D^i\cap \spann\bra{\set{v^j}_{j\in J^t_i}} +K^t +b^i$ and using $D^i=G^i\cap K^{s^i}$ we can use Proposition~\ref{generalorthprop} to describe $Q\bra{\mathcal{C}^{t}}$. Noting that $u^{i,j}=v^{i,j}$ if $s_j^i=-t_j$ and $u^{i,j}=0$ other wise we have that this description is equal to
\begin{subequations}
\begin{alignat}{3}
\gamma_{G^i}\bra{\sum\nolimits_{j\in \hat{J}_i^t} v^{i,j} \bra{v^{i,j}\cdot x-\sum\nolimits_{l=1}^n \overline{b}^{i,l}_j y_l }^{\hspace{-0.05in}\text{\relsize{2}{$+$}}\hspace{0.05in}}}&\leq y_i&\quad&\forall i\in \sidx{k}\label{isotonegeneralformproofgau}\\
t_j   v^j\cdot x-\sum\nolimits_{l=1}^n \underline{b}^{l}_j y_l &\geq 0 &\quad& \forall  j\in  \sidx{n}\label{isotonegeneralformprooflin}\\
 \sum\nolimits_{i=1}^k y_i=1,\quad y_i&\geq 0 &\quad& \forall i\in \sidx{k},
\end{alignat}
 \end{subequations}
 where  for all $i,l\in \sidx{k}$ and $j\in \sidx{n}$, $\hat{J}_i^t=\set{j\in J_i^t\,:\, s_j^i=-t_j}$ and
\begin{align*}
\underline{b}^{i}_j&=\min\set{t_j   v^j\cdot x\,:\, x\in b^i+D^i\cap \spann\bra{\set{v^j}_{j\in J^t_i}}}\\
&=\min\set{t_j   v^j\cdot x\,:\, x\in b^i+D^i\cap \spann\bra{\set{v^j}_{j\in J^t_i}}+K^t}\\
&=\min\set{t_j   v^j\cdot x\,:\, x\in b^i+D^i+K^t}=\min\set{t_j   v^j\cdot x\,:\, x\in b^i+D^i},
\end{align*}
where the first and last equality follow from $t_j v^j$ being a ray of $K^t$ and the second follows from the theorem's assumptions. Because $C^i=b^i+D^i$ we have that \eqref{isotonegeneralformprooflin} for all $t\in \set{-1,1}^n$ is equivalent to \eqref{isotonegeneralformlin}. To show that \eqref{isotonegeneralformproofgau} for all $t\in \set{-1,1}^n$ is equivalent to \eqref{isotonegeneralformgau} it suffices to note that if $\mu\in \Real^n_+$ and $\lambda\in \Real^n_+$ are such that $\mu_j\leq \lambda_j$ for all $j\in \sidx{n}$ then $\gamma_{G^i}\bra{\sum_{j=1}^n \mu_j v^{i,j}}\leq \gamma_{G^i}\bra{\sum_{j=1}^n \lambda_j v^{i,j}}$. For that assume for a contradiction that the reverse inequality holds for some $\mu$ and $\lambda$. Then we can scale $\mu$ and $\lambda$ so that $\sum_{j=1}^n \mu_j v^{i,j}\notin G^i$ and $\sum_{j=1}^n \lambda_j v^{i,j}\in G^i$. However, $\sum_{j=1}^n \lambda_j v^{i,j},\;\sum_{j=1}^n \mu_j v^{i,j}\in K^{s^i}$ so $\sum_{j=1}^n \mu_j v^{i,j}=\sum_{j=1}^n \lambda_j v^{i,j}-\sum_{j=1}^n (\lambda_j-\mu_j) v^{i,j}\in \bra{\bra{D^i-K^{s^i}}\cap K^{s^i}}$, which contradicts the the theorem's assumptions. The final statement by noting that  $\epi\bra{\gamma_{G^i}}=\set{\bra{x,y}\in \epi\bra{\gamma_{H^i}}\,:\,v^{i,j}\cdot \bra{x+b^i y}\leq \overline{b}^{i,i}_j y,\quad \forall j\in \sidx{n}}$.\qed
\end{proof}

Theorem~\ref{isotonegeneral} generalizes Theorem~\ref{isotone} in two ways. First by allowing unions of more than two sets. Second by relaxing the monotonicity requirement on the sets from a condition of the form $G^i_\infty= \Real^n_-$ to one of the form $\bra{G^i\cap \Real^n_+ -\Real^n_+}\cap\Real^n_+=G^i\cap \Real^n_+$. An example of a set that satisfies the latter condition, but not the former is the Euclidean ball. Theorem~\ref{isotonegeneral} achieves this by using a representation of the Minkowski sum based on the operation $\bra{\cdot}^+$ (cf. Lemma~\ref{finalfinalgaugelemma}), which can have some practical implications that we explore next.

\subsection{Minkowski sum, formulation size and constraint representation}

\blue{As noted in \cite{Hijazi,Hijazioptonline14}, ensuring formulation \eqref{isotoneform} from Theorem~\ref{isotonegeneral} is ideal may require adding all exponentially many (in $n$)   inequalities \eqref{isotoneformexp} for each $i\in \sidx{2}$.
 However,} formulation \eqref{isotonegeneralform} from Theorem~\ref{isotonegeneral} only requires one nonlinear inequality \eqref{isotonegeneralformgau} for each $i\in \sidx{k}$ to be ideal. We now study this seeming paradox starting with an example that shows how and when an exponential number of  inequalities \eqref{isotoneformexp} are needed.

\begin{example}\label{Example22} Let  $G^1:=\set{x\in \Real^n\,:\, \prod_{j=1}^n (2 - x_j) \geq 1, \; x_j\leq 2\; \forall j\in \sidx{n}}$, $G^2=-2\cdot{\bf 1}+\Real^n_+$, $r=2-2^{-1/n}<2-2^{1/(1-n)}$, $C^1=G^1\cap [0,r]^n$ and  $C^2=G^2\cap [-2,0]^n$. By Theorem~\ref{isotone} an ideal formulation for
$x\in C^1\cup C^2$ is given by
\begin{subequations}\label{oldisotoneexfor}
\begin{alignat}{3}
\gamma_{G^1}\bra{\brac{x}_{ J}}&\leq y_1&\quad&\forall J\subseteq \sidx{n}\label{oldisotoneexforexp}\\
-2 y_2  \leq x_j \leq r y_1,\quad \forall j\in \sidx{n},\quad
y_1+y_2=1,\quad y&\in \set{0,1}^2,\label{oldisotoneexforexpother}
 \end{alignat}
 \end{subequations}
 where we omitted $\gamma_{G^2}\bra{\brac{x}_{ J}}\leq y_2$ as they are redundant because $\epi\bra{\gamma_{G^2}}=\set{\bra{x,y}\,:\, x_i\geq -2 y \; \forall i\in \sidx{n}}$.
Alternatively, if for any $a \in \Real^n$ we let $\brac{a}^+\in \Real^n$ be such that ${\brac{a}^+}_j=\bra{a_j}^+$, then
by Theorem~\ref{isotonegeneral} an ideal formulation is given by
\begin{subequations}\label{newisotoneexfor}
\begin{alignat}{3}
\gamma_{G^1}\bra{\brac{x}^+ }&\leq y_1 &\quad&\label{newisotoneexforexp}\\
-2 y_2  \leq x_j \leq r y_1,\quad \forall j\in \sidx{n},\quad
y_1+y_2=1,\quad y&\in \set{0,1}^2,\label{newisotoneexforexpother}
 \end{alignat}
 \end{subequations}
 where we again removed a redundant inequality associated to $\gamma_{G^2}$. Finally,
 \begin{equation}\label{gaugerepofgeo}
 \hspace{-0.12in}\epi\bra{\gamma_{G^1}}=\set{\bra{x,y}\in \Real^{n}\times \Real_+: \begin{alignedat}{3}\prod_{j=1}^n (2y - x_j) \geq y^n,  x_j\leq 2y \, \forall j\in \sidx{n}\end{alignedat}}.
 \end{equation}
Now, by the selection of $r$ we have that for any $J\subseteq \sidx{n}$ such that $\abs{J}\leq n-1$, having $-2 y_2  \leq x_j \leq r y_1$ for all $j\in \sidx{n}$ implies
\[2^{n-\abs{J}} y^{n-\abs{J}} \prod\nolimits_{j\in J} \bra{2 y - x_j}> y^n 2^{\frac{n}{1-n}(\abs{J}-(n-1))}>y^n.\]
Hence replacing \eqref{oldisotoneexforexp} or \eqref{newisotoneexforexp} by $\gamma_{G^1}\bra{x}\leq y_1$
also yields an ideal formulation.

In contrast, if we instead let $r=2$, the  replacement of \eqref{oldisotoneexforexp} or \eqref{newisotoneexforexp}  results in a valid, but not ideal formulation.
Indeed, for any $J\subseteq \sidx{n}$ such that $\abs{J}\leq n-1$, let $\bra{x,y}\in \Real^{n+2}$ given by $y_1=y_2=1/2$, $x_j=1-2^{-n/\abs{J}}(3/2)^{\bra{\abs{J}-n}/J}$ for $j\in J$ and $x_j=-1/2$ for $j\notin J$. Then $\bra{x,y}$ is feasible for the continuous relaxation of \eqref{oldisotoneexforexpother}/\eqref{newisotoneexforexpother} and $\gamma_{G^1}\bra{x}\leq y_1$, but violates $\gamma_{G^1}\bra{\brac{x}_{ J}}\leq y_1$ and $\gamma_{G^1}(\brac{x}^+ )\leq y_1$. Hence, in this case formulation \eqref{oldisotoneexfor} from Theorem~\ref{isotone} requires an exponential number of inequalities, while formulation \eqref{newisotoneexfor} from Theorem~\ref{isotonegeneral} only requires a linear number of inequalities. However, the non-polyhedral nature of the inequalities makes such accounting a subtle matter. For instance, \eqref{oldisotoneexforexp} is equivalent to the single inequality $\max_{J\subseteq \sidx{n}}\gamma_{G^1}\bra{\brac{x}_{ J}}\leq y_1$ and in fact $\max_{J\subseteq \sidx{n}}\gamma_{G^1}\bra{\brac{x}_{ J}}=\gamma_{G^1}(\brac{x}^+)$. Then \eqref{oldisotoneexforexp} or \eqref{newisotoneexforexp} are different representations of the same convex constraint. Further insight into this can be gained by noting that  \eqref{newisotoneexforexp} (i.e. $\gamma_{G^1}(\brac{x}^+ )\leq y_1$) is equivalent to
\begin{equation}\label{extendedform}
\gamma_{G^1}\bra{z}\leq y_1, \quad x_j\leq z_j,\quad 0\leq z_j \quad \forall j\in \sidx{n}.
\end{equation}
Hence, \eqref{newisotoneexforexp} can be thought of as the implicit description of linear sized extended formulation \eqref{extendedform} of the exponential number of inequalities \eqref{oldisotoneexforexp}.
\qed
\end{example}
A detailed study of the size evaluation challenges illustrated in Example~\ref{Example22} is beyond the scope of this paper, but we make two observations about it.

The first concerns explicit formulation representations  that can be fed to a MIP solver. Formulation \eqref{newisotoneexforexp} can be explicitly represented using operation $(\cdot)^+$ or through extended formulation \eqref{extendedform}. The former could cause numerical issues due to the non-differentiability of $(\cdot)^+$, while the auxiliary variables $z$ of the latter could have a similar detrimental effect as variable copies $x^i$ of formulation \eqref{extendedformulation} from Theorem~\ref{extendedTheorem}. In addition, we can use representation \eqref{gaugerepofgeo} of $\gamma_{G^1}$ or use standard second order cone (SOC) representations of the geometric mean that use additional auxiliary variables (e.g. \cite{ben2001lectures}). In contrast to variable copies $x^i$, the auxiliary variables of such SOC representations have been shown to have a significant positive performance effect \cite{DBLP:conf/ipco/LubinYBV16,lubin2016polyhedral}. Hence, these implementation alternatives must be carefully compared to ensure the potential performance gain of the significantly smaller formulation \eqref{newisotoneexforexp} over \eqref{oldisotoneexforexpother} (or even formulations based on Theorem~\ref{extendedTheorem}) is achieved in practice. Similarly, the computational advantage of formulating trigonometric sets directly (as in  Example~\ref{quadrsticsin2}) instead of a quadratic approximation (as in Example~\ref{quadrsticsin2}) is uncertain because of the high quality of the approximation from \cite{bestuzheva2016convex,hijazi2013convex}.

The second observation concerns the existence of linear-sized formulations that do not use operation $(\cdot)^+$ or additional continuous auxiliary variables $z$. As noted in Example~\ref{Example22} this question is meaningless unless we give precise restrictions on the class of nonlinear inequalities we allow. Restricting to polynomial inequalities is not enough to achieve this goal, but the following example shows that it can still lead to interesting results and insights.

\begin{example}\label{zariex}The sets considered in Examples~\ref{example1}--\ref{example2b}, in Example~\ref{generalexamplecone} and in Example~\ref{quadrsticsin} can be described by a finite number of polynomial inequalities. Such sets are usually denoted \emph{basic semi-algebraic} and unions of such sets are usually denoted \emph{semi-algebraic} sets. It is known that the convex hull of the union of basic semi-algebraic sets is semi-algebraic, but not necessarily basic semi-algebraic. Hence, if $\mathcal{C}:=\set{C^i}_{i=1}^k$ is a finite family of basic semi-algebraic sets, then $Q\bra{\mathcal{C}}$ may or may not be basic semi-algebraic as it is the convex hull of particularly structured sets. The continuous relaxations of \eqref{idealofex1} and \eqref{generalconeex} show that $Q\bra{\mathcal{C}}$ is basic semi-algebraic for the sets in Examples~\ref{example1}--\ref{example2b} and in Example~\ref{generalexamplecone}. However, we now show that it is not basic semi-algebraic for the sets in Example~\ref{quadrsticsin}. For that take the affine section of the continuous relaxation of \eqref{notbasicform} obtained by fixing $y_1=y_2=1/2$ and which is given by
 \[
M:=\set{x\in \Real^2\,:\, \bra{\bra{(-1)^j x_1-1/2}^+}^2\leq x_2/2\quad  \forall j\in \sidx{2},\quad  x_2\leq 1/2}.
\]
This set is depicted in Figure~\ref{zariski} in gray where we can confirm that it is semi-algebraic (it is the convex hull of portions of two parabolas). However, we can check that the Zariski closure of its boundary (smallest algebraic variety that contains this boundary) is given  by \[Z:=\set{x\in \Real^2\,:\,  \biggl(\bra{x_1-\frac{1}{2}}^2-\frac{x_2}{2} \biggr)\biggl(\bra{x_1+\frac{1}{2}}^2-\frac{ x_2}{2}\biggr)\bra{x_2-\frac{1}{2}}x_2=0}\]
and depicted in black in Figure~\ref{zariski}. We can also check that $Z\cap \Int\bra{M}\neq \emptyset$, which is a known impediment for a set to be  basic semi-algebraic \cite{andradas1994ubiquity,blekherman2013semidefinite}. \qed
\end{example}
\begin{figure}[htb]
\centering \includegraphics[scale=.45]{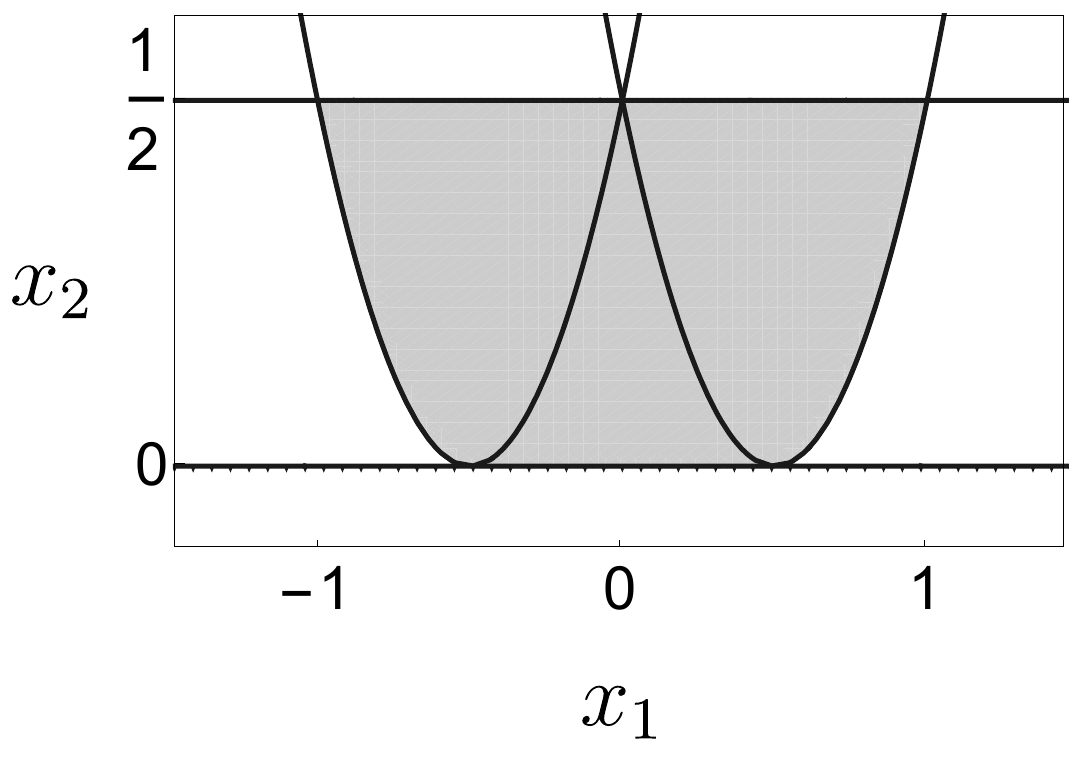}
\caption{Set $M$ from Example~\ref{zariex} and its Zariski closure.}\label{zariski}
\end{figure}

Note that for the sets in Examples~\ref{example1}--\ref{example2b} and  in Example~\ref{generalexamplecone} the description of the Minkoswki sum from Lemma~\ref{commontangent2} does  not require the operation $(\cdot)^+$ and $Q\bra{\mathcal{C}}$ is basic semi-algebraic.  In contrast, the operation is required for Example~\ref{quadrsticsin} and $Q\bra{\mathcal{C}}$ is not basic semi-algebraic. This shows that operation $(\cdot)^+$ can affect the properties of $Q\bra{\mathcal{C}}$ and that this is strongly tied to the Minkoswki sum operation. In fact, using Proposition~\ref{embpart2} below, Example~\ref{zariex} yields $C^1:=[-1,1]\times\set{0}$ and $C^2:=\set{x\in [-1,1]\times [0,1]\,:\, x_1^2\leq x_2}$ as examples of  basic semi-algebraic sets whose Mikowski sum is not basic semi-algebraic.

\section{Necessary and Sufficient Conditions for Piecewise Formulations}\label{necandsufsec}

Example~\ref{generalexamplecone} shows how condition \eqref{par2} of Proposition~\ref{newtheoref} may not be necessary to obtain an ideal formulation. We now give necessary and sufficient strength conditions  through a variant of \eqref{par1} that guarantees formulation validity.

\begin{definition}\label{generalformnonid}Let $\mathcal{C}:=\set{C^i}_{i=1}^k\in \mathbb{C}_n$ and $\mathcal{C}^j:=\set{C^{j,i}}_{i=1}^k\in \mathbb{C}_n$ for $j\in \sidx{m}$ be such that $C^i= \bigcap_{j=1}^m C^{j,i}$ for all $i\in \sidx{k}$ so that a valid formulation of $x\in \bigcup_{i=1}^k C^i$ is given by
\begin{equation}\label{sharpidealform}
\bra{x,y}\in Q\bra{\mathcal{C}^{j}}\quad \forall j\in \sidx{m},\quad \sum\nolimits_{i=1}^k y_i,\quad y\in \set{0,1}^k.
\end{equation}
We say \eqref{sharpidealform} is ideal if its continuous relaxation is equal to $Q\bra{\mathcal{C}}$ and \emph{sharp} if the projection of this relaxation onto the $x$ variables is equal to $\conv\bra{\bigcup_{i=1}^k C^i}$.
\end{definition}

Being sharp is a weaker strength requirement than being ideal (e.g. by Proposition~\ref{embpart2} below, if  \eqref{sharpidealform} is ideal, then it is sharp), but can still result in good computational performance (e.g. \cite[Section 2.2]{Mixed-Integer-Linear-Programming-Formulation-Techniques}). In fact, the polyhedral work of Balas, Blair and Jeroslow \cite{balas88,blair90,jeroslow88} considered in Section~\ref{polysec}  focused on constructing sharp formulations and resulted in necessary conditions that can be stated in the context of Definition~\ref{generalformnonid} as follows.

\begin{theorem}[Theorem 3 in \cite{blair90}]\label{necblairteo} Let $A\in \mathbb{R}^{m\times n}$ and for each $i\in \sidx{k}$ let $b^i\in \Real^m$ and $P^i=\set{x\in \Real^m\,:\,Ax\leq b^i}$, for each $B\in \mathcal{B}$ let $\mathcal{C}^B:=\set{C^{B,i}}_{i=1}^k$ be such that $C^{B,i}=P\bra{B,b^i}$ for all $i\in \sidx{k}$ and $\set{\mathcal{C}^j}_{j=1}^m=\set{\mathcal{C}^B}_{B\in \mathcal{B}}$. If \eqref{sharpidealform} is not sharp then \blue{there exists $B\in \mathcal{B}$, $i_1,i_2\in \sidx{k}$ and $u\in \Real^n$ such that
\begin{equation*}
\sigma_{\conv\bra{\bigcup_{i=1}^k C^i}}\bra{u}=\sigma_{C^{B,i_1}}\bra{u}\quad\text{and}\quad
																								\sigma_{\conv\bra{\bigcup_{i=1}^k C^i}}\bra{u}<\sigma_{C^{B,i_2}}\bra{u}.
\end{equation*}}
\end{theorem}
To extend  Theorem~\ref{necblairteo} we use the following generalization to  non-polyhedral sets of a known relation between the Cayley embedding and the Minkowski sum of polytopes \cite{caytrick,karavelas2013maximum,WeibelPhd}. We present a proof in Section~\ref{proofofbdprop}.
\begin{restatable}{proposition}{embpartproptwo}\label{embpart2}
Let $\Delta^k:=\{\lambda\in \Real^k_+\,:\, \sum_{i=1}^k \lambda_i =1\}$, $\mathcal{C}:=\set{C^i}_{i=1}^k\in \mathbb{C}_n$, $Q\subseteq \Real^{n+k}$ be a closed convex set such that $Q\subseteq \Real^n\times \Delta^k$ and
\beq\label{begingformulation}\bra{x,y}\in Q \cap \bra{\Real^n\times\mathbb{Z}^k} \quad \Leftrightarrow \quad\exists i\in \set{1,\ldots,k} \text{ s.t. } y=\e^i\;\wedge\; x\in C^i,\eeq
and $Q\bra{\bar{y}}:=\set{x\in \Real^n: \bra{x,\bar{y}}\in Q}$. Then $Q\bra{\mathcal{C}}\subseteq Q$, $C_\infty=Q\bra{\mathcal{C}}_\infty$,
\beq\label{minkwoskioneway}
\forall \bar{y}\in \Delta^k\quad \quad \bar{x}\in \sum\nolimits_{i=1}^k \bar{y}_i C^i \quad\Rightarrow \quad\bra{\bar{x},\bar{y}}\in Q.
\eeq
and the following are equivalent
\begin{enumerate}
\item\label{minkowski1} $Q=Q\bra{\mathcal{C}}$.
\item\label{minkowski2} $\forall \bar{y}\in \Delta^k\quad \quad \bra{\bar{x},\bar{y}}\in Q  \quad\Rightarrow \quad \bar{x}\in \sum_{i=1}^k \bar{y}_i C^i$.
\item\label{minkowski4} $\bra{\bar{x},\frac{1}{k}\mathbf{1}}\in Q  \quad\Rightarrow \quad \bar{x}\in \frac{1}{k}\sum_{i=1}^k  C^i$.
 \setcounter{nameOfYourChoice}{\theenumi}
 \end{enumerate}
\end{restatable}
\begin{theorem}\label{newtheorefnew} For the $\set{\mathcal{C}^j}_{j=1}^m$ from Definition~\ref{generalformnonid} and for any $\lambda\in \Delta^k$ let $Q^m\bra{\lambda}:=\bigcap_{j=1}^m \sum_{i=1}^k \lambda_i C^{j,i}$. Formulation \eqref{sharpidealform} is sharp if and only if
\[\max\nolimits_{i=1}^k \sigma_{C^i}\bra{u}=\max\nolimits_{\lambda\in \Delta^k}\sigma_{Q^m\bra{\lambda}}\bra{u} \quad \forall u\in \Real^n.\]
Formulation \eqref{sharpidealform} is ideal if and only if
\[\sum\nolimits_{i=1}^k \lambda_i \sigma_{C^i}\bra{u}=\sigma_{Q^m\bra{\lambda}}\bra{u} \quad \forall \lambda\in \Delta^k,\;u\in \Real^n,\]
or equivalently if and only if
\[\bra{1/k}\sum\nolimits_{i=1}^k \sigma_{C^i}\bra{u}=\sigma_{Q^m\bra{\bra{1/k}{\mathbf 1}}}\bra{u}\quad \forall u\in \Real^n.\]
Finally, the equivalences can be written as a function of $\sigma_{C^{j,i}}$ by noting that
\[\sigma_{Q^m\bra{\lambda}}\bra{u}=\liminf_{\bar{u}\to u} \inf\set{\sum\nolimits_{j=1}^m \sum\nolimits_{i=1}^k \lambda_i \sigma_{C^{j,i}}\bra{u^j}\,:\,\sum\nolimits_{j=1}^m u^j=\bar{u}}.\]
In particular, formulation \eqref{sharpidealform} is ideal if and only if for all $u\in \Real^n$
\begin{equation}\label{generalidealcondition}
\hspace{-0.1in}\sum\nolimits_{i=1}^k \sigma_{C^i}\bra{u}=\liminf_{\bar{u}\to u} \inf\set{\sum\nolimits_{j=1}^m \sum\nolimits_{i=1}^k \sigma_{C^{j,i}}\bra{u^j}\,:\,\sum\nolimits_{j=1}^m u^j=\bar{u}}.
\end{equation}
\end{theorem}
\begin{proof}
By Theorem C.3.3.2 in \cite{hiriart-lemarechal-2001}  we obtain the characterization of $\sigma_{Q^m\bra{\lambda}}$ and that for all $u\in \Real^n$ we have that $\sigma_{\conv\bra{\bigcup_{i=1}^k C^i} }\bra{u}=\max_{i=1}^k \sigma_{C^i}\bra{u}$, $\sigma_{ \sum_{i=1}^k \lambda_i C^i}\bra{u}=\sum_{i=1}^k \lambda_i \sigma_{C^i}\bra{u}$ and $\sigma_{\bigcup_{\lambda\in \Delta^k} Q^m\bra{\lambda}}\bra{u}=\max_{\lambda\in \Delta^k} \sigma_{Q^m\bra{\lambda}}\bra{u}$.

The result for being sharp follows from Proposition~\ref{embpart2} implying $\bigcap_{j=1}^m Q\bra{\mathcal{C}^j}=\bigcup_{\lambda\in \Delta^k} Q^m\bra{\lambda} \times \set{\lambda}$ and hence that its projection onto the $x$ variables is $\bigcup_{\lambda\in \Delta^k} Q^m\bra{\lambda}$. The results for being ideal follow from Proposition~\ref{embpart2} implying that $\bigcap_{j=1}^m Q\bra{\mathcal{C}^j}=Q\bra{\mathcal{C}}$ if and only if $Q^m\bra{\lambda}=\sum_{i=1}^k \lambda_i C^i$ for all $\lambda\in \Delta^k$ or equivalently if $Q^m\bra{\bra{1/k}{\mathbf 1}}=(1/k)\sum_{i=1}^k C^i$.\qed
\end{proof}
The conditions for formulation \eqref{sharpidealform} being ideal and sharp from Theorem~\ref{newtheorefnew} can be contrasted by noting that for all $u\in \Real^n$
\[\max\nolimits_{i=1}^k \sigma_{C^i}\bra{u}=\max\nolimits_{\lambda\in \Delta^k} \sum\nolimits_{i=1}^k \lambda_i\sigma_{C^i}\bra{u}\geq \sum\nolimits_{i=1}^k(1/k)\sigma_{C^i}\bra{u}.\]
Hence, being sharp requires matching the maximum weighted average of the support functions while being ideal requires matching all weighted averages or equivalently the equal weight average or simply the sum.

The necessary and sufficient condition \eqref{generalidealcondition} for being ideal of Theorem~\ref{newtheorefnew} can in turn be contrasted with condition \eqref{par2} of Proposition~\ref{newtheoref} which requires
\[\forall u\in \Real^n\quad \exists j\in\sidx{m} \text{ s.t. }\sigma_{C^i}\bra{u}=\sigma_{C^{j,i}}\bra{u}\; \forall i\in \sidx{k}.\]
For instance, condition \eqref{generalidealcondition} can be simplified to replace condition \eqref{par2} with the slightly weaker condition
\blue{
\beq\label{ex5cond}\begin{alignedat}{3}\forall u\in \Real^n\; \exists \set{u^j}_{j=1}^m\subseteq \Real^n \; &\text{ s.t. }\\ u	=\sum\nolimits_{j=1}^m u^j \quad\text{ and }&\quad\sigma_{C^i}\bra{u}=\sum\nolimits_{j=1}^m\sigma_{C^{j,i}}\bra{u^j}\quad \forall i\in \sidx{k}.\end{alignedat}\eeq
}
We can check that sets $\{\mathcal{C}^j\}_{j=1}^2$ in the first part of Example~\ref{generalexamplecone} satisfy condition \eqref{ex5cond}, but only if we add recession cone $C^{j,0}_\infty$ following Lemma~\ref{commontangent2} (cf. the left side of Figure~\ref{gcfig1} where the dotted curve describes $C^{1,1}$ if we do not add the cone). Similarly to the comments after Example~\ref{example2b}, one way to interpret the need to satisfy condition \eqref{ex5cond} for Example~\ref{generalexamplecone} is to ensure that there is a non-zero intersection of the normals to $\{C^{1,i}\}_{i=1}^2$ at the portions of the boundary highlighted in white in Figure~\ref{gcfig1}. The following corollary formalized this idea into a sufficient condition that can be useful to verify that formulation \eqref{sharpidealform} is ideal and/or to guide the construction of $\{\mathcal{C}^j\}_{j=1}^m$ to obtain an ideal formulation.
\begin{corollary}\label{alternativetheo}Let the $\set{\mathcal{C}^j}_{j=1}^m$ from Definition~\ref{generalformnonid} be such that  $\aff\bra{Q\bra{\mathcal{C}}}\subseteq \bigcap_{j=1}^m Q\bra{\mathcal{C}^{j}}$, and for $\mathcal{D}:=\set{\mathcal{D}^i}_{i=1}^k$ in $\Real^n$ let $L\bra{\mathcal{D}}:=\sum_{i=1}^k L\bra{D^i}$ and
$N\bra{\mathcal{D}}:=\set{\bra{{x}^i}_{i=1}^k\in \cart\nolimits_{i=1}^k \bd\bra{D^i}  \,:\, L\bra{\mathcal{D}}\cap\bigcap\nolimits_{i=1}^k N_{D^i}\bra{{x}^i} \neq \set{\bf 0}}$.
Then formulation \eqref{sharpidealform} is ideal if and only if
\begin{equation}\label{generalnormalconecond}
N\bra{\mathcal{C}}\subseteq \bigcup\nolimits_{j=1}^m N\bra{\mathcal{C}^j}.
\end{equation}
\end{corollary}
\begin{proof}
We have that $Q\bra{\mathcal{C}}=\bigcup_{j=1}^m Q\bra{\mathcal{C}^j}$ if and only if their affine hulls and relative boundaries match. Under the assumptions we have $Q\bra{\mathcal{C}}\subseteq \bigcap_{j=1}^m Q\bra{\mathcal{C}^j}$, $\aff\bra{Q\bra{\mathcal{C}}}\subseteq \bigcap_{j=1}^m Q\bra{\mathcal{C}^{j}}$ and $y_i\geq 0$ for all $i\in \sidx{n}$ and $\bra{x,y}\in  \bigcap_{j=1}^m Q\bra{\mathcal{C}^j}$. Hence, $Q\bra{\mathcal{C}}=\bigcup_{j=1}^m Q\bra{\mathcal{C}^j}$ if and only if portion \eqref{relbdunion2} of the boundary characterization of $Q\bra{\mathcal{C}}$ from Proposition~\ref{bdprop} is equal to the union of the same portions for the $Q\bra{\mathcal{C}^j}$, which is equivalent to \eqref{generalnormalconecond}.
\qed
\end{proof}
We can check that sets $\{\mathcal{C}^j\}_{j=1}^2$ in the first part of Example~\ref{generalexamplecone} also satisfy condition \eqref{generalnormalconecond} and redundant sets $\mathcal{C}^3$ are not needed to show formulation \eqref{generalconeex} is ideal. Now, in this case, the redundancy of $\mathcal{C}^3$ needed for Proposition~\ref{newtheoref} only resulted in easy to recognize duplicate inequalities in \eqref{generalconeex}. However, the following example shows how using Corollary~\ref{alternativetheo} instead of Proposition~\ref{newtheoref} can avoid more consequential redundancies.
\begin{example}\label{Example66}Consider again the sets from Example~\ref{Example22} given by $C^1=G^1\cap [0,r]^n$ for $G^1:=\set{x\in \Real^n\,:\, \prod_{j=1}^n (2 - x_j) \geq 1, \; x_j\leq 2\; \forall j\in \sidx{n}}$, and $C^2=[-2,0]^n$. The first version of these sets takes $r=2-2^{-1/n}$ and is depicted in Figure~\ref{3da} for $n=3$.
The redundancy analysis in the example yielded the simplified version of the formulation from Theorems~\ref{isotone} and \ref{isotonegeneral} for $x\in C^1\cup C^2$ given by
\begin{equation}\label{simpleformbd}
\gamma_{G^1}\bra{x}\leq y_1,\; -2 y_2  \leq x_j \leq r y_1,\; \forall j\in \sidx{n},\;
y_1+y_2=1,\; y\in \set{0,1}^2.
 \end{equation}
An alternative way to get this formulation is by noting that the boundary of $C^1$ has a polyhedral portion associated to the variable bounds and a non-polyhedral portion associated to $G^1$. This non-polyhedral portion is highlighted dark gray in Figure~\ref{3da1} for $n=3$, and for all $n$ it can be sub-divided into $\bd\bra{G^1}\cap\bra{0,r}^n$  and $\bd\bra{G^1}\cap\bd\bra{[0,r]^n}$. If $x^1\in \bd\bra{G^1}\cap\bra{0,r}^n$ we have that $N_{C^1}\bra{x^1}$ is contained in the strictly positive orthant. Hence for all $x^1\in \bd\bra{G^1}\cap\bra{0,r}^n$ we have $\bra{x^1,x^2}\in N\bra{\mathcal{C}}$ if and if $x^2={\bf 0}$, which is also  highlighted in dark gray in Figure~\ref{3da2}. In contrast, because of the choice of $r$ we have that if $x^1\in \bd\bra{G^1}\cap\bd\bra{[0,r]^n}$ then there exist $J\subseteq \sidx{n}$ such that $x^1\in \set{x\in \bd\bra{G^1}\,:\, x_j=r\quad \forall j\in J}$ and $\bra{x^1,x^2}\in N\bra{\mathcal{C}}$ if and if $x^2 \in \bigcup_{j\in J} \set{x\in [-2,0]^n\,:\, x_j=0}$. Then condition \eqref{generalnormalconecond} is satisfied for $\set{\mathcal{C}^j}_{j=1}^2$ given by  $C^{1,1}:=[0,r]^n$, $C^{1,2}:=[-2,0]^n$, $C^{2,1}=G^1$, $C^{2,2}=\Real^n_-$. In particular, the key for satisfying the condition is that for all $\bra{x^1,x^2}\in N\bra{\mathcal{C}}$ such that $x^1\in \bd\bra{G^1}\cap\bd\bra{[0,r]^n}$  we have that $\bra{x^1,x^2}\in N\bra{\mathcal{C}^1}$.  Finally, Corollary~\ref{alternativetheo} with this decomposition yields precisely \eqref{simpleformbd}.\qed
\end{example}
\begin{figure}[htpb]
\centering
\subfigure[Set $C^1$.]{\includegraphics[scale=0.3]{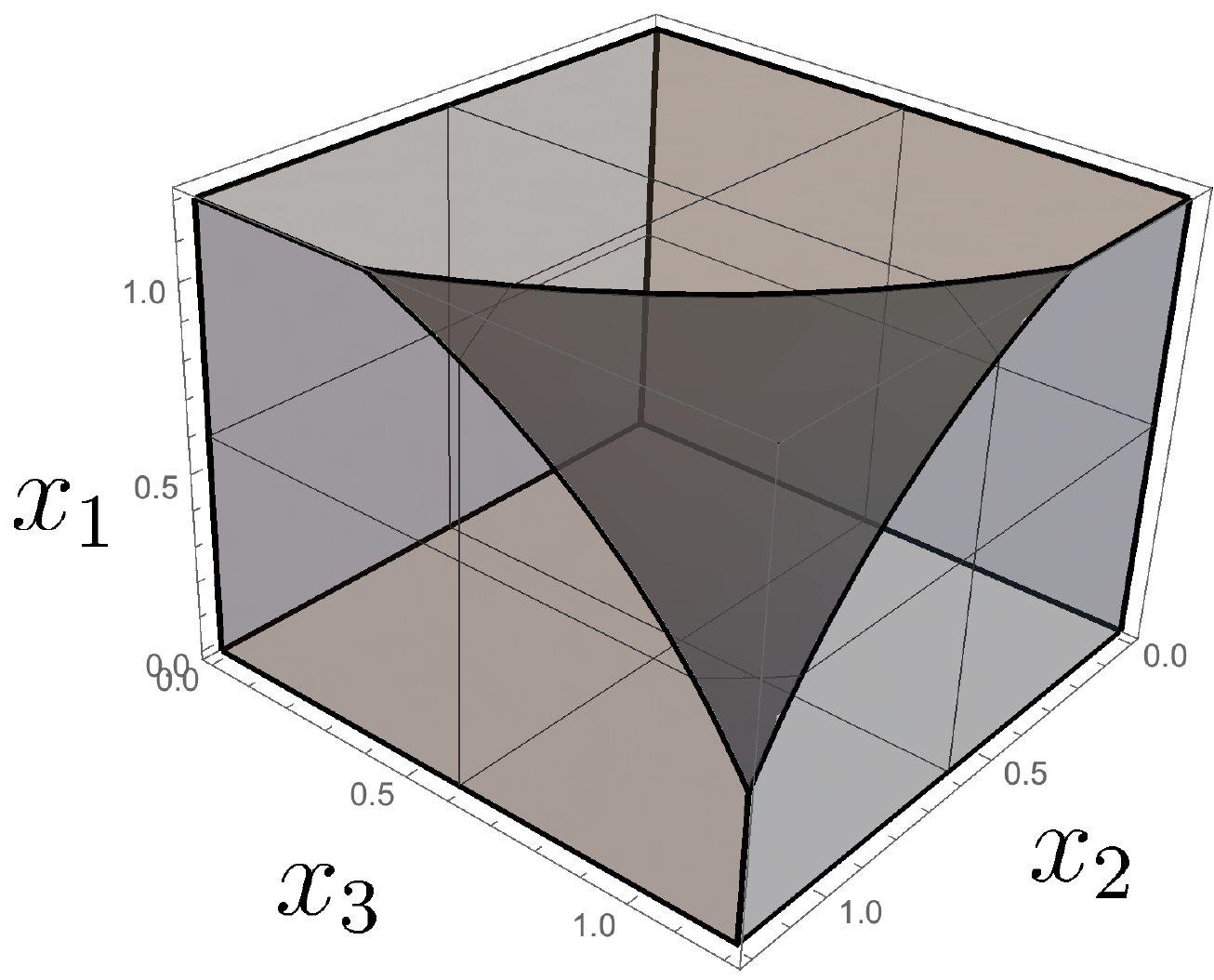}\label{3da1}}\quad\quad\quad
\subfigure[Set $C^2$.]{\includegraphics[scale=0.4]{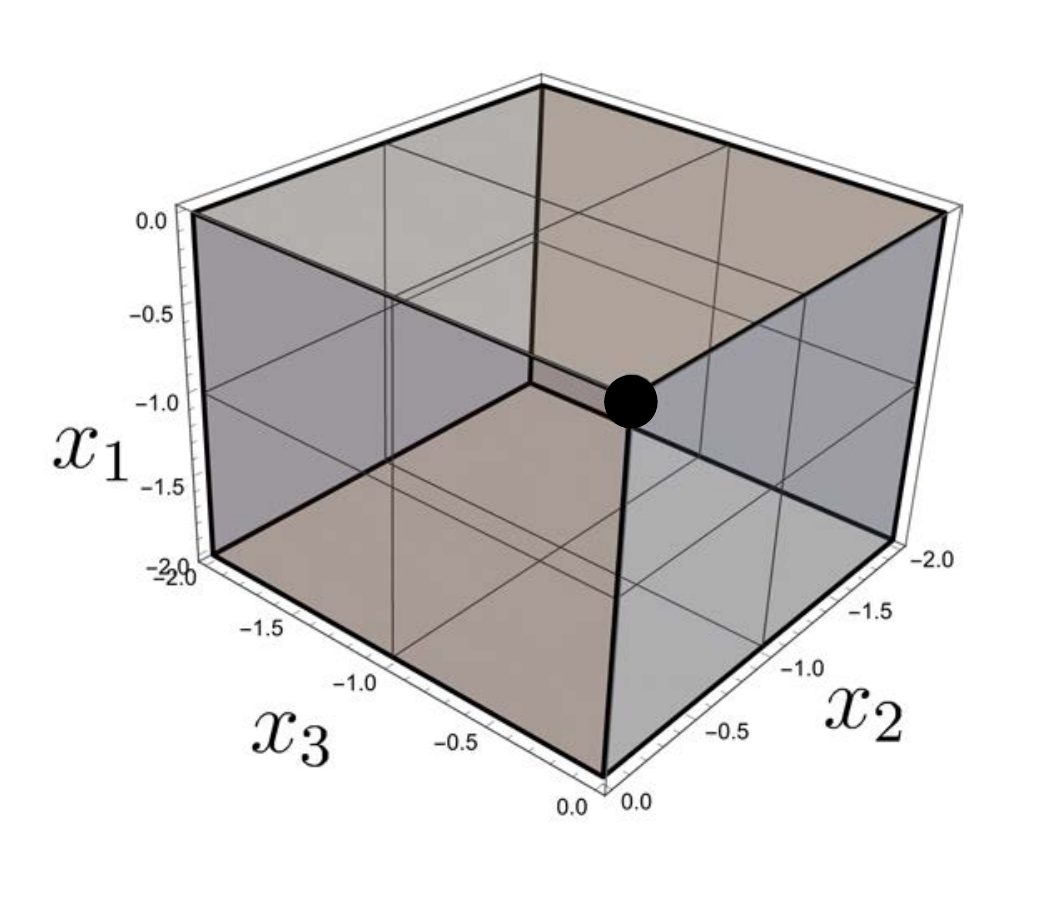}\label{3da2}}
\caption{Sets for Example~\ref{Example66}.}\label{3da}
\end{figure}

Our final example illustrates how Corollary~\ref{alternativetheo} can be used to show Theorem~\ref{isotone} and give a geometric interpretation of the associated formulation.

\begin{example}\label{Example33}Consider now the second version of the sets from Example~\ref{Example22} which corresponds to the same sets in Example~\ref{Example66}, but with $r=2$. These sets depicted in Figure~\ref{3db} for $n=3$. We can again use $\set{\mathcal{C}^j}_{j=1}^2$ given by  $C^{1,1}:=[0,r]^n$, $C^{1,2}:=[-2,0]^n$, $C^{2,1}=G^1$, $C^{2,2}=\Real^n_-$
to get valid formulation \eqref{simpleformbd}. However, from Example~\ref{Example22}  we know that for this choice of $r$ this formulation is no longer ideal. Indeed, condition \eqref{generalnormalconecond} of Corollary~\ref{alternativetheo} is no longer  satisfied because we no longer have $\bra{x^1,x^2}\in N\bra{\mathcal{C}^1}$ for all $\bra{x^1,x^2}\in N\bra{\mathcal{C}}$ such that $x^1\in \bd\bra{G^1}\cap\bd\bra{[0,s]^n}$. For instance, if $x^1\in D^1:=\bd\bra{G^1}\cap\bra{(0,s)^{n-1}\times \set{0}}$ (highlighted in dark gray in Figure~\ref{3db1}) and $x^2\in D^2:=\bd\bra{C^2}\cap\bra{\set{0}^{n-1}\times[-r,0]}$ (highlighted in dark gray in Figure~\ref{3db2}) we have that $\bra{x^1,x^2}\in N\bra{\mathcal{C}}$, but $x^1\notin \bd\bra{C^{1,1}}$. This specific case can be resolved by adding $\mathcal{C}^4$ such that $C^{4,1}:=G^1+\spann\bra{\set{\e^n}}$ and $C^{4,2}:=G^1_\infty+\spann\bra{\set{\e^n}}=(-\infty,0]^{n-1}\times \set{0}$ ($C^{4,1}$ is depicted in Figure~\ref{3db1} by the transparent meshed surface), as $\bra{x^1,x^2}\in N\bra{\mathcal{C}^4}$ for all $\bra{x^1,x^2}\in D^1\times D^2$. Similarly, we can resolve all additional cases and satisfy condition \eqref{generalnormalconecond} by adding $\mathcal{C}^J$ such that $C^{J,1}:=G^1+\spann\bra{\set{\e^j}_{j\in J}}$ and $C^{J,2}:=G^1_\infty+\spann\bra{\set{\e^j}_{j\in J}}$ for all $J\subseteq \sidx{n}$ with $\abs{J}\leq n-1$. By noting that $\gamma_{C^{J,1}}\bra{x}=\gamma_{G^1}\bra{\brac{x}_{ J}}$, we have that the formulation obtained from Corollary~\ref{alternativetheo} for $\set{\mathcal{C}^J}_{J\subseteq \sidx{n}}$ is precisely formulation \eqref{oldisotoneexfor} obtained from Theorem~\ref{isotone}.\qed
\end{example}
\begin{figure}[htpb]
\centering
\subfigure[Set $C^1$.]{\includegraphics[scale=0.45]{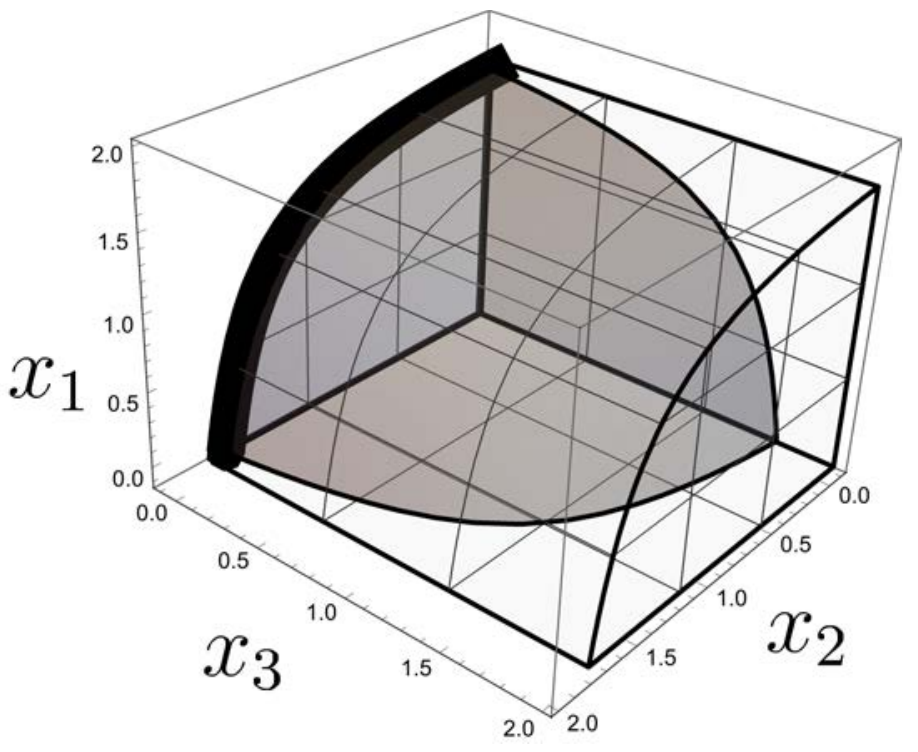}\label{3db1}}\quad\quad\quad
\subfigure[Set $C^2$.]{\includegraphics[scale=0.42]{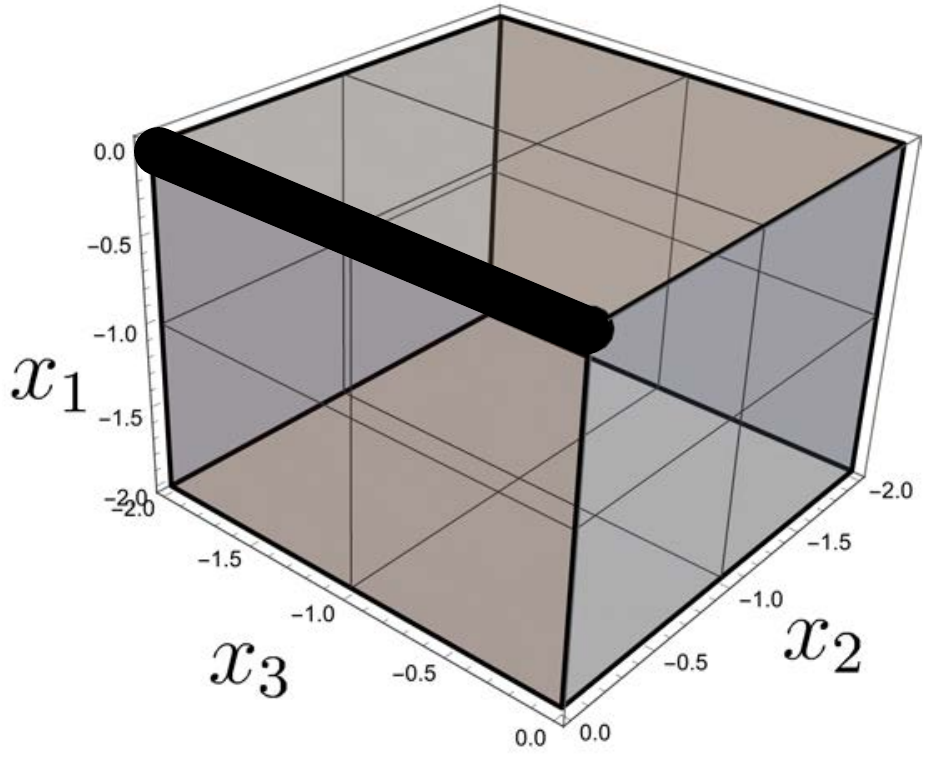}\label{3db2}}
\caption{Sets for Example~\ref{Example33}.}\label{3db}
\end{figure}

\blue{
\section{Conclusions}

Modeling disjunctive constraints with ideal MIP formulations that avoid the variable copies of standard convex hull formulations can provide a computational advantage over both convex hull and Big-M formulations \cite{lodi15,Hijazi}. Unfortunately, existing techniques to construct such formulations are restricted to special structures or require ad-hoc decomposition techniques. In this paper, we  introduced systematic and generic construction tools for these formulations. The tools do require understanding the geometry of the disjunctive constraints. However,  when this understanding is available, the techniques are easily applicable even for high-dimensional constraints, constraints with a large number of terms, and highly non-polyhedral constraints (e.g. Example~\ref{generalexamplecone}). The resultant formulations can usually be represented in a format compatible with MIP solvers through simple \emph{gauge-calculus} (e.g. Lemma~\ref{gaugelemma}). However, these representations may include maximum operations that could lead to differentiability issues (e.g. Examples~\ref{quadrsticsin}, \ref{quadrsticsin2} and \ref{Example22}). Such  issues can be avoided through standard linear programming tricks, but these introduce continuous auxiliary variables that are similar to the variable copies we aimed to avoid. Nonetheless, these auxiliary variables may not necessarily  have the same negative computational effect as the variable copies. Finally, as illustrated in Example~\ref{Example22},  the max operation and the continuous auxiliary variables can sometimes be avoided with little or no loss of formulation strength, and, as illustrated in \cite{lodi15,Hijazi}, even when some strength is lost, the  formulations can still provide an advantage.
}
\section{Omitted Proofs}\label{proofsec}

\subsection{Theorem \ref{extendedTheorem} and Proposition~\ref{embpart}}\label{Moreongaugesec}

\begin{proof}[of Theorem \ref{extendedTheorem}]
Validity is direct from Lemma~\ref{gaugepropertiesprop}, 	$C^i_\infty=\bra{C^i-a^i}_\infty$ and $\mathcal{C}\in \mathbb{C}_n$. For idealness, let $Q$ be the continuous relaxation of \eqref{extendedformulation} and assume for a contradiction that there exist a minimal face $F$ of $Q$  and   $(x,y)\in F$  with $y\notin \set{0,1}^k$. Without loss of generality $y_1,y_2\in (0,1)$.  Let $\varepsilon=\min\{y_1,y_2,1-y_1,1-y_2\}\in (0,1)$, $\underline{y}_1=\overline{y}_2=y_1+\varepsilon$, $\underline{y}_2=\overline{y}_1=y_2-\varepsilon$, $\underline{y}_i=\overline{y}_i=y_i$ for all  $i\notin\set{1,2}$,
	$\underline{x}^i=(\underline{y}_i/y_i)x^i$ and $\overline{x}^i=(\overline{y}_i/y_i)x^i$  for $i\in\set{1,2}$, $\underline{x}^i=\overline{x}^i=x^i$  for all  $i\notin\set{1,2}$, $\underline{x}=\sum_{i=1}^k\underline{x}^i$  and $\overline{x}=\sum_{i=1}^k\overline{x}^i$. Then $(\underline{x},\underline{y})\neq(\overline{x},\overline{y})$, $(x,y)=(1/2)(\underline{x},\underline{y})+(1/2)(\overline{x},\overline{y})$. Multiplying $\gamma_{C^i-b^i}\bra{x^i-b^i y_i}\leq y_i$ for $i\in \set{1,2}$ by $\underline{y}_i/y_i$ or $\overline{y}_i/y_i$ and using the positive homogeneity of $\gamma_{C^i-b^i}$ we have   $(\underline{x},\underline{y}), (\overline{x},\overline{y})\in Q$. Hence, $(\underline{x},\underline{y}), (\overline{x},\overline{y})\in F$. Furthermore,  by construction either $\underline{y}_1=1$, $\underline{y}_2=0$, $\overline{y}_1=0$ or $\overline{y}_2=1$. If $\underline{y}_1=1$, then $\set{\bra{x,y}\in F\,:\, y_1=1}\subsetneq F$ is a face of the continuous relaxation of \eqref{extendedformulation}, which contradicts the minimality of $F$. All other three cases are analogous. 	The final statement follows from the recession cone of the continuous relaxation of \eqref{extendedformulation} being equal to all $\bra{x,y}$ such that $x\in C^1_\infty$, $y=0$ and $x^i\in C^1_\infty$ for all $i\in \sidx{k}$. \qed
\end{proof}

\begin{proof}[of Proposition~\ref{embpart}]
Part \ref{embpart1} follow directly from Corollary 9.8.1 in \cite{rockafellar2015convex} by noting that if $\set{C^i}_{i=1}^k\in \mathbb{C}_n$
 then $\set{C^i\times\set{e^i}}_{i=1}^k\in \mathbb{C}_{n+1}$. For part  \ref{embpart3} note that we have that $\bra{x,y}\in Q\bra{\mathcal{C}}$ if and only if $y\in \Real^k_+$, $\sum_{i=1}^k y_i=1$ and
\beq\label{sumcond1}
\exists \bra{\tilde{x}^i}_{i=1}^k\in \cart\nolimits_{i=1}^k C^i \text{ s.t. } x=\sum\nolimits_{i=1}^k y_i \tilde{x}^i.
\eeq
The result follows directly if   \eqref{sumcond1} is equivalent to
\beq\label{sumcond2}
\exists \bra{x^i}_{i=1}^k\in \Real^{n\cdot k} \text{ s.t. } x=\sum\nolimits_{i=1}^k  x^i \text{ and } \gamma_{C^i-b^i}\bra{x^i-b^i y_i}\leq y_i.
\eeq
 To show this equivalence first note that $\tilde{x}^i\in C^i$ if and only if  $\gamma_{C^i-b^i}\bra{\tilde{x}^i-b^i }\leq 1$, and if $y_i>0$ this last condition is in turn equivalent to $\tilde{x}^i=x^i/y_i$ for some $x^i\in \Real^n$ such that $\gamma_{C^i-b^i}\bra{x^i-b^i y_i}\leq y_i$. Then note that if $y_i=0$, then $\gamma_{C^i-b^i}\bra{x^i-b^i y_i}\leq y_i$ if and only if $x^i\in C^i_\infty$. To show that \eqref{sumcond1} implies \eqref{sumcond2} simply let $x^i=y_i \tilde{x}^i $. For the reverse implication assume without loss of generality that $y_1>0$ and let $I_0=\set{i\in \sidx{k}\,:\, y_i=0}$. Then  $\tilde{x}^1=x^1/y_1 +\sum_{i\in I_0} x^i\in C^1$ by  $\set{C^i}_{i=1}^k\in \mathbb{C}_n$. Finally, the implication follows because $\tilde{x}^i=x^i/y_i\in C^i $ for all $i\in \sidx{k}\setminus I_0$. Part \ref{embpart44} follows directly from part  \ref{embpart3}.	\qed
\end{proof}

\subsection{Proof of Proposition~\ref{generalorthprop}}\label{proofofgeneralorthpropsec}

\begin{lemma}\label{finalfinalgaugelemma}Let $C\subseteq \Real^n$ be a closed convex set containing $\bf 0$, $\set{v^j}_{j=1}^n\subseteq \Real^n$ be an orthonormal basis of $\Real^n$, $s\in \set{-1,0,1}^n$, $t\in \set{-1,1}^n$, $K=\cone\bra{\set{ s_j v^j}_{j=1}^n}$, $M=\cone\bra{\set{ t_j v^j}_{j=1}^n}$ and
 $u^j=(-s_j t_j)^+ s_j v^j$ for all $j\in \sidx{n}$.
If $C\cap K$ is compact and  $\bra{\bra{C\cap K}-K}\cap K=C\cap K$ then $\bra{x,y}\in\epi\bra{\gamma_{C\cap K+M}}$ if and only if
\begin{equation}\label{gaugeconechar}
\gamma_C\bra{\sum\nolimits_{j=1}^n u^j \bra{u^j\cdot x}^+ }\leq y,\; ((1-\abs{s_j})+(s_j t_j)^+)t_j  v^j\cdot x \geq 0 \; \forall j\in  \sidx{n}.\end{equation}
\end{lemma}
\begin{proof} Let $E$ be the region described by \eqref{gaugeconechar}. We have that $E$ is a closed convex cone such that $y\geq 0$ for all $\bra{x,y}\in E$ so by Lemma~\ref{gaugelemma} we just need to show that $\bra{x,0}\in E$ is equivalent to $x\in \bra{ C\cap K+M}_\infty=M$ and that $\bra{x,1}\in E$ is equivalent to $x\in  C\cap K+M$.

For the first implication of both equivalence let $y\in \set{0,1}$, $C_1=C$, $C_0=C_\infty$, $\bra{x,y}\in E$, $J=\set{j\in \sidx{n}\,:\, u^j\cdot x> 0}$, $x^C=\sum_{j\in J} v^j v^j\cdot x$ and $x^M=\sum_{j\in \sidx{k}\setminus J} v^j v^j\cdot x$. Because $\set{v^j}_{j=1}^n\subseteq \Real^n$ is an orthonormal basis we have $x=x^C+x^M$. Furthermore, $\sum\nolimits_{j=1}^n u^j (u^j\cdot x)^+=\sum\nolimits_{j=1}^n v^j v^j\cdot x^C=x^C$ so $x^C\in C_y$, and $s_j v^j \cdot x^C >0$ for all $j\in J$ and $s_j v^j \cdot x^C =0$ for all $j\in \sidx{n}\setminus J$ so $x^C\in K$. Finally, if $j \in \sidx{n}\setminus J$, then either (i) $s=-t$ and $t_j v^j\cdot x\geq 0$, or (ii) $s\in\set{0,t}$. In the second case the linear inequalities of \eqref{gaugeconechar}
imply  $t_j v^j\cdot x\geq 0$. Then $t_j v^j\cdot x^M\geq 0$ for all $j \in \sidx{n}\setminus J$ and $t_j v^j\cdot x^M= 0$ for all $j \in  J$. Hence, $x^M\in M$ and for $y=1$ we have $x\in C\cap K+M$. Similarly for  $y=0$ we have $x\in C_\infty\cap K+M= \bra{C\cap K}_\infty+M=M$ (cf. Proposition A.2.2.5 in \cite{hiriart-lemarechal-2001}).

For the reverse implication of the first equivalence note that if $x\in M$, then  $\bra{x,0}\in E$ because $x$ satisfies the linear inequalities of  \eqref{gaugeconechar} and $\bra{u^j\cdot x}^+=0$ for all $j\in \sidx{n}$. For the second equivalence let $x=x^C+x^M$ with $x^C\in C\cap K$ and $x^M\in M$. Then $x$ satisfies the linear inequalities of  \eqref{gaugeconechar}  because both $x^C$ and $x^M$ satisfy them. Now let $J:=\set{j\in \sidx{n}\,:\, s_j=-t_j,\; s_jv^j\cdot\bra{x^C+x^M}> 0}$ and
$\tilde{x}:=\sum\nolimits_{j\in \sidx{n}\setminus J} v^j v^j\cdot x^C \quad$ $+\sum\nolimits_{j\in J} v^j\bra{-v^j x^M}$.
By the definition of $J$ and because $x^M\in M$ and $x^C\in K$ we have $\tilde{x}\in K$ and $x^C-\tilde{x}=\sum_{j\in J} v^j v^j\cdot\bra{x^C + x^M}\in K$. Then $x^C-\tilde{x}\in \bra{ \bra{C\cap K}-K}\cap K$ and hence by the assumption on $C$ and $K$ we have $x^C-\tilde{x}\in C$. Then $x$ satisfies non-linear inequality of  \eqref{gaugeconechar}  because $J=\set{j\in \sidx{n}\,:\, u^j\cdot \bra{x^C + x^M}> 0}$ and hence
$x^C-\tilde{x}=\sum_{j\in J} v^j v^j\cdot\bra{x^C + x^M}=\sum\nolimits_{j=1}^n u^j (u^j\cdot \bra{x^C + x^M})^+$.\qed
\end{proof}

\begin{proof}[of Proposition~\ref{generalorthprop}]
The result will follow from Proposition~\ref{embpart} by showing that \eqref{orthogonalplusprojcone} is the projection of the continuous relaxation of \eqref{extendedformulation} for the considered sets.
Noting that $(-0 t)^+ 0=0 $ for all $t\in \set{-1,1}$ we can use Lemma~\ref{finalfinalgaugelemma} to show that the continuous relaxation of  \eqref{extendedformulation} is given by
 \begin{subequations}\label{orthogonalplusprojconeext}
\begin{alignat}{3}\gamma_{G^i}\bra{\sum\nolimits_{j\in J_i} u^{i,j} \bra{u^{i,j}\cdot \bra{x^i-b^i y_i} }^+}&\leq y_i&\quad&\forall i\in \sidx{k}\\  ((1-\abs{s_j^i})+(s_j^i t_j)^+)t_j v^j\cdot \bra{x^i-b^i y_i} &\geq 0 &\quad& \forall i\in \sidx{k}\,\; j\in  \sidx{n}\label{btobarb}\\
 \sum\nolimits_{i=1}^k x^i=x,\quad\sum\nolimits_{i=1}^k y_i=1,\quad y_i &\geq 0 &\quad& \forall i\in\sidx{k}.
\end{alignat}\end{subequations}
Now, for all $i\in \sidx{k}$ and $j\in \sidx{n}$ such that $s_j=0$ or $s_j=t_j$ we have $\underline{b}^{i}_j=t_j v^j\cdot b^i$, so \eqref{btobarb} is dominated by $t_j v^j\cdot x^i-\underline{b}^{i}_j y_i \geq 0$ for all $i\in \sidx{k}$ and $j\in  \sidx{n}$ (the additional inequalities for case $s_j=t_j$ are clearly valid). To show that \eqref{orthogonalplusprojcone} is contained in the projection of \eqref{orthogonalplusprojconeext}  let $\bra{x,y}$ be feasible for \eqref{orthogonalplusprojcone} and for all $i\in \sidx{k}$ let  $\lambda_j^i:=y_i t_j \underline{b}^{i}_j$ if $j\in \sidx{n}\setminus J_i$ and $\lambda_j^i=v^j\cdot x-\sum_{l\in \sidx{k}\setminus \set{i}} \lambda_j^l$ if $j\in J_i$. Finally, for all $i\in\sidx{k}$ let $x^i=\sum_{j=1}^n \lambda_j^i v^j$. We can check that $\bra{x,\bra{x^i}_{i=1}^k,y}$ is feasible for \eqref{orthogonalplusprojconeext}. In particular, $\bra{x,y}$ feasible for \eqref{orthogonalplusprojcone1} implies $\bra{x^i,y_i}$ is feasible for \eqref{btobarb} because $u^{i,j}=0$ if $s^i_j\neq -t_j$ and if $s^i_j= -t_j$, then $\bra{-s_j^i t_j}^+s_j^it_j=-1$, $\underline{b}_j^l=-\bar{b}^{i,l}_j$, and hence $u^{i,j}\cdot \bra{x^i-b^i y_i}=u^{i,j}\cdot (x- b^i y_i)+\sum_{l\in \sidx{k}\setminus \set{i}}y_i  \underline{b}^{i}_j=u^{i,j}\cdot x-\sum\nolimits_{l=1}^k \overline{b}^{i,l}_j y_l $.
 The reverse inclusion follows from validity of \eqref{orthogonalplusprojcone} plus $y\in \mathbb{Z}^k$ as a formulation for $x\in \bigcup_{i=1}^k C^i$.\qed
 \end{proof}

\subsection{Proof of Proposition~\ref{bdprop}}\label{proofofbdprop}

\begin{proposition}\label{knwonbdlemma} For a closed convex set $C\subseteq \Real^n$ we have that $\relbd\bra{C}=\bigcup_{d\in D\cap L(C)} F_C\bra{d}$ for $D=L(C)\setminus \set{0}$ or  $D=L(C)\cap\dom\bra{\sigma_C}\setminus \set{0}$. In addition, if $u\in \Real^n$ and $w\in L(C)^\perp$, then $F_C(u)=F_C(u-w)$.
\end{proposition}
\begin{proof}
The proof of the first statement is identical to that of Proposition C.3.1.5 in \cite{hiriart-lemarechal-2001}. For the second note that by Definition C.2.1.4 and Proposition C.1.1.7 we have that $\sigma_C(w)=\sigma_C(-w)$ and $\sigma_C(u-w)=\sigma_C(u)+\sigma_C(-w)$. Furthermore, $-w\cdot x=\sigma_C(-w)$ for all $x\in C$. Then for any $x\in C$ we have
\begin{alignat*}{3}x\in F_C(u) &\Leftrightarrow u\cdot x =\sigma_C(u)\Leftrightarrow u\cdot x -w\cdot x=\sigma_C(u)+\sigma_C(-w) \\&\Leftrightarrow \bra{u-w}\cdot x=\sigma_C(u-w)\Leftrightarrow x\in F_C(u-w).\hspace{1in}\qed\end{alignat*}
\end{proof}

\begin{lemma}\label{bdpropauxlemma} Let $\mathcal{C}:=\set{C^i}_{i=1}^k\in \mathbb{C}_n$, $u\in \Real^n$ and $v\in \Real^k$. Then $\sigma_{Q\bra{\mathcal{C}}}\bra{u,v}=\max_{i=1}^k \sigma_{C^i}\bra{u}+ v\cdot e^i$
and $F_{Q\bra{\mathcal{C}}}\bra{{u,v}}=\conv\bra{\bigcup_{i\in I\bra{u,v}} F_{C^i}\bra{u}\times e^i }$ for $I\bra{u,v}:=\set{i\in \sidx{k}\,:\, \sigma_{C^i}\bra{u}+ v\cdot e^i=\sigma_{Q\bra{\mathcal{C}}}\bra{{u,v}}}$.
\end{lemma}
\begin{proof} The characterization of $\sigma_{Q\bra{\mathcal{C}}}\bra{u,v}$ is direct from Theorem C.3.3.2 in \cite{hiriart-lemarechal-2001}. For the characterization of the face of $Q\bra{\mathcal{C}}$ exposed by $\bra{u,v}$ note that $\bra{x,y}\in F_{Q\bra{\mathcal{C}}}\bra{{u,v}}$ if and only if there exist $\lambda\in \Delta^k:=\set{\lambda\in \Real^k_+\,:\, \sum_{i=1}^k \lambda_i =1}$  and $x^i\in C^i$ for $i\in \sidx{k}$ such that $x=\sum_{i=1}^k \lambda_i x^i$, $y=\sum_{i=1}^k \lambda_i e^i$ and
\beq\label{facelemma1}
u\cdot x+v\cdot y= \sum\nolimits_{i=1}^k \lambda_i \bra{u\cdot x^i + v\cdot e^i} = \sigma_{Q\bra{\mathcal{C}}}\bra{u,v}.
\eeq
By the definition of $\sigma_{C^i}$ and the characterization of $\sigma_{Q\bra{\mathcal{C}}}$ for all $i\in \sidx{k}$
\beq\label{facelemma2}
u\cdot x^i + v\cdot e^i \leq \sigma_{C^i}\bra{u} +v\cdot e^i \leq \sigma_{Q\bra{\mathcal{C}}}\bra{u,v}.
\eeq
So if $\lambda_i>0$ in \eqref{facelemma1} for $i\in \sidx{k}$ then both inequalities in \eqref{facelemma2} hold as equalities for $i$. Then  \eqref{facelemma1} holds if and only if for all $i\in \sidx{k}$ with $\lambda_i>0$ we have (i) $u\cdot x^i=\sigma_{C^i}\bra{u}$ or equivalently $x^i\in F_{C^i}\bra{u}$, and (ii) $\sigma_{C^i}\bra{u}+ v\cdot e^i=\sigma_{Q\bra{\mathcal{C}}}\bra{{u,v}}$.\qed\end{proof}

\begin{proof}[of Proposition \ref{bdprop}]\label{proofofbdprop} Let $ E:=\{\bra{x,y}\in \Real^{n+k}\,:\, \sum\nolimits_{i=1}^k y_i=1,\quad Ax=\sum\nolimits_{i=1}^k A b^i y_i\}$. The inclusion $\aff\bra{Q\bra{\mathcal{C}}}\subseteq E$ follows by noting that  $\aff\bra{C^i}\subseteq \set{x\in \Real^n\,:\, Ax=Ab^i}$ and hence $\bra{x,\e^i}\in E$ for all $x\in C^i$. For the reverse inclusion let $\bra{x,y}\in E$ and $\bar{x}=\bra{x-\sum_{i=1}^k b^i y_i}$. Then $A\bar{x}={\bf 0}$, so $\bar{x}\in L\bra{\mathcal{C}}$ and hence there exist $\bar{x}^i\in L\bra{C^i}$ for $i\in \sidx{k}$ such that $\bar{x}=\sum_{i=1}^k \bar{x}^i$. For any $i\in \sidx{k}$ and $\lambda\neq 0$ we have $\bar{x}^i/\lambda\in L\bra{C^i}$, $\bar{x}^i/\lambda+b^i\in \aff\bra{C^i}$ and hence $\bra{\bar{x}^i/\lambda+b^i,\e^i}\in \aff\bra{C^i\times \set{\e^i}}\subseteq \aff\bra{Q\bra{\mathcal{C}}}$. In particular, for any any $i\in \sidx{k}$ we have $\bra{\bar{x}^i,{\bf 0}}=\bra{\bar{x}^i/2+b^i,\e^i}-\bra{-\bar{x}^i/2+b^i,\e^i}\in L\bra{Q\bra{\mathcal{C}}}$ and if $y_i\neq 0$ we have $\bra{\bar{x}^i/y_i+b^i,\e^i}\in \aff\bra{Q\bra{\mathcal{C}}}$. Then, letting $I_0=\set{i\in \sidx{k}\,:\, y_i=0}$ and $I_1=\sidx{n}\setminus I_0$ we have $x=\sum_{i\in I_1} y_i \bra{\bar{x}^i/y_i+b^i,\e^i} +\sum_{i\in I_0} \bra{\bar{x}^i,{\bf 0}} \in \aff\bra{Q\bra{\mathcal{C}}}$.

Then by Proposition~\ref{knwonbdlemma} we have $\relbd\bra{Q\bra{\mathcal{C}}}= \bigcup\nolimits_{\bra{u,v}\in L\bra{\mathcal{C}}\setminus\set{0}}F_{Q\bra{\mathcal{C}}}\bra{u,v}$.
The result will follow by refining the right hand side of this inclusion to include only the $F_{Q\bra{\mathcal{C}}}\bra{u,v}$ that are maximal with respect to inclusion.

We begin by showing that \eqref{relbdunion1} corresponds to the maximal faces when $u \in \mathcal{U}\bra{\mathcal{C}}$. Indeed, from Lemma~\ref{bdpropauxlemma} we only need to show that for all $\bar{u} \in \mathcal{U}\bra{\mathcal{C}}$ there exist $\bra{u,v}\in L\bra{\mathcal{C}}\setminus\set{0}$ such that $I\bra{u,v}=\sidx{k}$ and $F_{C^i}\bra{\bar{u}}=F_{C^i}\bra{u}$ for all $i\in \sidx{k}$. For that first let  $\bar{v}\in \Real^k$ be such that $\bar{v}_1= -\frac{1}{k}\sum_{i=2}^k \bra{\sigma_{C^1}(\bar{u})-\sigma_{C^i}(\bar{u})}$ and $\bar{v}_j= \sigma_{C^1}(\bar{u})-\sigma_{C^j}(\bar{u})  -\frac{1}{k}\sum_{i=2}^k \bra{\sigma_{C^1}(\bar{u})-\sigma_{C^i}(\bar{u})}$ for all $j\in \sidx{k}\setminus \set{1}$. Then, $I\bra{\bar{u},\bar{v}}=\sidx{k}$ and $\sum_{i=1}^k \bar{v}_i=0$. If $\bar{u}-\sum\nolimits_{i=1}^k  b^i \bar{v}_i\in L\bra{\mathcal{C}}$ we are done by letting $\bra{u,v}=\bra{\bar{u},\bar{v}}$. If not, there exist $w\in L\bra{\mathcal{C}}^\perp$ such that $u-\sum\nolimits_{i=1}^k  b^i v_i\in L\bra{\mathcal{C}}$ for $u=\bar{u}-w$ and $v=\bar{v}$. Now, for any $i\in \sidx{k}$ we have $w\in L\bra{\mathcal{C}}^\perp\subseteq L\bra{C^i}^\perp $ and hence by Proposition~\ref{knwonbdlemma} we have $F_{C^i}\bra{\bar{u}}=F_{C^i}\bra{\bar{u}-w}=F_{C^i}\bra{u}$. In particular, for all $i\in \sidx{k}$ there exist $x^i\in F_{C^i}\bra{u}$ such that  $\sigma_{C^i}(\bar{u}-w)=x^i\cdot \bra{\bar{u}-w} $ and $\sigma_{C^i}(\bar{u})=x^i\cdot \bar{u}$. Then   ${\sigma_{C^1}(\bar{u})-\sigma_{C^i}(\bar{u})}={\sigma_{C^1}(\bar{u}-w)-\sigma_{C^i}(\bar{u}-w)}$ for all $i\in \sidx{n}$ and hence $I\bra{{u},{v}}=\sidx{k}$ and $\sum_{i=1}^k {v}_i=0$ by the definition of $\bar{v}=v$ and $u=\bar{u}-w$.

We can also check that \eqref{relbdunion2} corresponds to the maximal faces exposed by $\bra{0,v}$ for $v\in\Real^k$, which are precisely those exposed when there exist $i\in \sidx{n}$ such that  $v_i=1-k$  and $v_j=1$ for $i\neq j$.

The last case is $u\in L\bra{\mathcal{C}}\setminus \set{0}$ and there exist $\emptyset\neq I\subseteq \sidx{k}$ such that $F_{C^i}\bra{u}=\emptyset$ for  $i\in I$ and $F_{C^i}\bra{u}\neq \emptyset$ for  $i\in \sidx{n}\setminus I$. An analog argument to  case $u \in \mathcal{U}\bra{\mathcal{C}}$ shows that the maximal faces here correspond to $\bra{u,v}\in L\bra{\mathcal{C}}\setminus\set{0}$ such that $I\bra{u,v}=\sidx{n}\setminus I$. However, those faces are contained in $ \conv\bra{\bigcup\nolimits_{j\neq i} C^j \times \set{\e^j}}$ for any $i\in I$, which are already included in \eqref{relbdunion2}.

The alternative characterizations for \eqref{relbdunion1}/\eqref{relbdunion2} follow from the fact that  $x\in F_C\bra{u}$ if and only if $u\in N_C\bra{x}$ (e.g. Proposition C.3.1.4 in \cite{hiriart-lemarechal-2001}).\qed
\end{proof}

\subsection{Proof of Proposition~\ref{embpart2}}

\begin{proof}[of Proposition~\ref{embpart2}]
Property \eqref{begingformulation} implies $\bigcup\nolimits_{i=1}^k C^i\times \set{e^i}\subseteq Q$ which shows $Q\bra{\mathcal{C}}\subseteq Q$,  $Q\bra{\mathcal{C}}_\infty \subseteq Q_\infty$ and \eqref{minkwoskioneway}. $Q\subseteq \Real^n\times \Delta^k$ implies  $Q_\infty \subseteq \Real^n\times \set{0}$ and $Q\bra{e^i}=C^i\times \set{e^i}$ further implies that $Q_\infty \subseteq C^1_\infty\times \set{0}=Q\bra{\mathcal{C}}_\infty$.

Part \ref{minkowski1} implies Part \ref{minkowski2} is direct from the definition of $Q\bra{\mathcal{C}}$, which together with \eqref{minkwoskioneway} shows their equivalence. Part \ref{minkowski2} implies \ref{minkowski4}  is direct.

For  \ref{minkowski4} implies \ref{minkowski1} we show that if $Q\bra{\mathcal{C}}\subsetneq Q$, then there exist $\tilde{x}\in \Real^n$ such that $\bra{\tilde{x},\frac{1}{k}\mathbf{1}}\in Q $ and $\tilde{x}\notin \frac{1}{k} \sum_{i=1}^k C^i$. For this we first claim that if $\bra{\bar{x},\bar{y}}\in Q\setminus Q\bra{\mathcal{C}}$ then there exist $a\in \Real^n$, $b\in \Real^k$ and $c\in \Real$ that satisfy the following three separation conditions: (i) $a\cdot \bar{x}+b\cdot \bar{y}> c$, (ii)$a\cdot x+b\cdot y\leq  c$ for all $\bra{x,y}\in Q\bra{\mathcal{C}}$, and (iii) for all  $i\in \sidx{k}$ and $\varepsilon>0$ there exist $\bar{x}^i\bra{\varepsilon}\in C^i$ such that $a\cdot \bar{x}^i\bra{\varepsilon}+b\cdot e^i\geq c-\varepsilon$. Indeed the first two  follow from the separation theorem for closed convex sets. If the third condition does not hold for some $i\in \sidx{k}$ then $\max\set{a\cdot x\,:\, x\in C^i}< c- b_i$ and because $\bar{y}\geq 0$ we can decrease $b_i$ to achieve the equality while still satisfying the first two conditions.

Now, because of \eqref{minkwoskioneway} for
$Q=Q\bra{\mathcal{C}}$ and separation condition (ii) we have
\beq\label{coldeq} a\cdot x+(1/k)\sum\nolimits_{i=1}^k b_i\leq  c \quad \forall x\in (1/k) \sum\nolimits_{i=1}^k C^i. \eeq
Additionally, because $\bar{y}\in \Delta^k$ there exist $\bra{\lambda_0,\lambda}\in \Delta^{k+1}$ with $\lambda_0>0$ such that $\lambda_0 \bar{y}+\sum_{i=1}^k \lambda_i e^i= \frac{1}{k} \mathbf{1}$. If  $\sum_{i=1}^k \lambda_i=0$, then $\lambda_0=1$, $\bar{y}=\frac{1}{k}\mathbf{1}$ and $\bra{\bar{x},\frac{1}{k}\mathbf{1}}\in Q $.  Hence, because of separation condition (i) and \eqref{coldeq} we have $\bar{x}\notin \frac{1}{k} \sum_{i=1}^k C^i$. If instead we have  $\sum_{i=1}^k \lambda_i>0$, then there exist $\varepsilon>0$ such that $\lambda_0\bra{a\cdot \bar{x}+b\cdot \bar{y}-c}/\bra{\sum_{i=1}^k \lambda_i}>\varepsilon$ because of separation condition (i). For such $\bra{\lambda_0,\lambda}$ and $\varepsilon$ let $\bra{\tilde{x},\tilde{y}}=\lambda_0 \bra{\bar{x},\bar{y}}+\sum_{i=1}^k \lambda_i \bra{\bar{x}^i\bra{\varepsilon},e^i}$. Because $ \bra{\bar{x}^i\bra{\varepsilon},e^i}\in Q$ for each $i\in \sidx{k}$ we then have that $\bra{\tilde{x},\frac{1}{k}\mathbf{1}}\in Q $. Furthermore, because separation conditions (i) and (iii), and the condition on $\varepsilon$ we have $a\cdot \tilde{x}+\frac{1}{k}\sum_{i=1}^k b_i > c$ and hence by \eqref{coldeq} we have $\tilde{x}\notin \frac{1}{k} \sum_{i=1}^k C^i$.
\qed
\end{proof}

\begin{acknowledgements}
This research was partially supported by NSF under grant CMMI-1351619. \blue{We thank two anonymous referees for their constructive comments that improved the paper's presentation.}
 \end{acknowledgements}

\bibliographystyle{spmpsci}

\end{document}